\pdfoutput=1
\RequirePackage{ifpdf}
\ifpdf 
\documentclass[pdftex]{sigma}
\else
\documentclass{sigma}
\fi

\numberwithin{equation}{section}

\newtheorem{Theorem}{Theorem}[section]
\newtheorem*{Theorem*}{Theorem}
\newtheorem*{Question*}{Question}
\newtheorem{Corollary}[Theorem]{Corollary}
\newtheorem{Lemma}[Theorem]{Lemma}
\newtheorem{Proposition}[Theorem]{Proposition}

 { \theoremstyle{definition}
\newtheorem{Definition}[Theorem]{Definition}

\newtheorem{Example}[Theorem]{Example}
\newtheorem{Remark}[Theorem]{Remark} }

\def\SL{\mathrm{SL}}

\def\N{\mathbb N}
\def\C{\mathbb C}

\def\Q{\mathbb Q}
\def\R{\mathbb R}

\def\Z{\mathbb Z}
\def\mod{\ \mathrm{mod}\ }

\def\SL{\mathrm{SL}}

\def\PSU{\mathrm{PSU}}

\def\wt#1{\widetilde{#1}}
\def\M#1#2#3#4{\begin{pmatrix}#1&#2\\#3&#4\end{pmatrix}}
\def\SM#1#2#3#4{\left(\begin{smallmatrix}#1&#2\\#3&#4\end{smallmatrix}
 \right)}

\def\X #1,#2{X^{#1}_0(#2)}

\def\bt{\boldsymbol t}

\begin{document}

\allowdisplaybreaks

\newcommand{\arXivNumber}{2205.13912}

\renewcommand{\PaperNumber}{040}

\FirstPageHeading

\ShortArticleName{Co-Axial Metrics on the Sphere and Algebraic Numbers}

\ArticleName{Co-Axial Metrics on the Sphere\\ and Algebraic Numbers}

\Author{Zhijie CHEN~$^{\rm a}$, Chang-Shou LIN~$^{\rm b}$ and Yifan YANG~$^{\rm c}$}

\AuthorNameForHeading{Z.~Chen, C.-S.~Lin and Y.~Yang}

\Address{$^{\rm a)}$~Department of Mathematical Sciences, Yau Mathematical Sciences Center,\\
\hphantom{$^{\rm a)}$}~Tsinghua University, Beijing, 100084, P.~R.~China}
\EmailD{\href{mailto:zjchen2016@tsinghua.edu.cn}{zjchen2016@tsinghua.edu.cn}}

\Address{$^{\rm b)}$~Department of Mathematics, National Taiwan University, Taipei 10617, Taiwan}
\EmailD{\href{mailto:cslin@math.ntu.edu.tw}{cslin@math.ntu.edu.tw}}

\Address{$^{\rm c)}$~Department of Mathematics, National Taiwan University\\
\hphantom{$^{\rm c)}$}~and National Center for Theoretical Sciences, Taipei 10617, Taiwan}
\EmailD{\href{mailto:yangyifan@ntu.edu.tw}{yangyifan@ntu.edu.tw}}

\ArticleDates{Received November 07, 2023, in final form May 09, 2024; Published online May~20, 2024}

\Abstract{In this paper, we consider the following curvature equation
 \begin{gather*}
 \Delta u+{\rm e}^u=4\pi\biggl((\theta_0-1)\delta_0+(\theta_1-1)\delta_1
 +\sum_{j=1}^{n+m}\bigl(\theta_j'-1\bigr)\delta_{t_j}\biggr)\qquad \text{in}\ \mathbb R^2, \\
 u(x)=-2(1+\theta_\infty)\ln|x|+O(1)\qquad \text{as} \ |x|\to\infty,
 \end{gather*}
 where $\theta_0$, $\theta_1$, $\theta_\infty$, and $\theta_{j}'$ are positive non-integers for $1\le j\le n$, while $\theta_{j}'\in\mathbb{N}_{\geq 2}$ are integers for $n+1\le j\le n+m$. Geometrically, a solution $u$ gives rise to a conical metric ${\rm d}s^2=\frac12 {\rm e}^u|{\rm d}x|^2$ of curvature $1$ on the sphere, with conical singularities at $0$, $1$, $\infty$, and $t_j$, $1\le j\le n+m$, with angles $2\pi\theta_0$, $2\pi\theta_1$, $2\pi\theta_\infty$, and
 $2\pi\theta_{j}'$ at $0$, $1$, $\infty$, and $t_j$, respectively. The metric ${\rm d}s^2$ or the solution $u$ is called {\it co-axial}, which was introduced by Mondello and Panov, if there is a developing map $h(x)$ of $u$ such that the projective monodromy group is contained in the unit circle. The sufficient
 and necessary conditions in terms of angles for the existence of such metrics were obtained by Mondello--Panov (2016) and Eremenko (2020). In this paper, we fix the angles and study the locations of the singularities $t_1,\dots,t_{n+m}$. Let $A\subset\mathbb{C}^{n+m}$ be the set of those $(t_1,\dots,t_{n+m})$'s such that a co-axial metric exists, among other things we prove that (i)~If $m=1$, i.e., there is only one integer $\theta_{n+1}'$ among~$\theta_j'$, then~$A$ is a finite set. Moreover, for the case $n=0$, we obtain a sharp bound of the cardinality of the set $A$. We apply a result due to Eremenko, Gabrielov, and Tarasov (2016) and the monodromy of hypergeometric equations to obtain such a bound. (ii)~If $m\ge 2$, then $A$ is an algebraic set of dimension $\leq m-1$.}


\Classification{57M50}

\section{Introduction}

Let $\bigl(\mathbb{S}^2,g_0\bigr)$ be the $2$-dimensional
sphere with the standard smooth metric $g_0$.
In this paper, we consider the following classical problem in conformal geometry:
Given $t_1,\dots,t_N\in \mathbb{S}^2$ and a set
of positive real numbers $\theta_{t_1},\dots,\theta_{t_N}$, is there
 a metric of constant curvature $1$ conformal to $g_0$
such that the
metric is smooth on $\mathbb{S}^2\setminus\{t_1,\dots,t_N\}$ and has conical
singularities at $t_j$ with the angle $2\pi\theta_{t_j}$? This problem has been widely studied in the literature, and in particular, it
was solved in \cite{Eremenko,Coaxial,EGT} in the cases where there are at
most $2$ non-integer angles among
$\{\theta_{t_1},\dots,\theta_{t_N}\}$ \big(We simply say that the angle
is non-integer if $\theta_j\notin\mathbb{Z}$\big). In this paper, we
consider the case that there are at least $3$ non-integer angles among
$\{\theta_{t_1},\dots,\theta_{t_N}\}$. Without loss of generality, we
may assume that three singularities with non-integer angles are $0$,
$1$, and $\infty$.
Then it is well known that this problem is equivalent to solving
 the following curvature
equation
\begin{gather}
 \Delta u+{\rm e}^u=4\pi\Biggl(\alpha_0\delta_{0}+\alpha_1\delta_{1}
 +\sum_{j=1}^{n+m}\alpha_{t_j}\delta_{t_j}\Biggr)
 \qquad \text{in }\mathbb R^2\cong\C,\nonumber \\
 u(x)=-2(2+\alpha_\infty)\ln|x|+O(1) \qquad \text{as }|x|\to\infty, \label{00eq: DE on C}
 \end{gather}
where $\delta_p$ is the Dirac measure at $p$, $n\geq 0$, $m\geq 1$, $\alpha_0, \alpha_1, \alpha_\infty,
\alpha_{t_j}\in\R_{>-1}\setminus\mathbb{Z}$ for all $1\leq j\leq n$
and $\alpha_{t_j}\in\mathbb{N}$ for all $n+1\leq j\leq
m$. Geometrically, a solution $u(x)$ of (\ref{00eq: DE on C}) leads to
a conformal metric ${\rm d}s^2=\frac12{\rm e}^u|{\rm d}x|^2$ of
curvature $1$ on $\mathbb{S}^2$ with angles $2\pi\theta_{p}$ at conical
singularities $p\in\{0,1,\infty, t_j\}$, where the angles are given by
\begin{equation*} 
 \theta_{p}=\alpha_p+1, \qquad
 p\in \{0,1,\infty, t_1,\dots,t_{n+m}\},
\end{equation*}
so
\begin{gather*} 
 \theta_{p}\in \mathbb{N}_{\geq 2}, \quad
 p\in \{t_{n+1},\dots,t_{n+m}\},\qquad
 \theta_p\in \R_{>0}\setminus\mathbb{Z},\quad p\in\{0,1,\infty,t_1,\dots,t_n\}.
\end{gather*}
It is known that if~\eqref{00eq: DE on C} has a solution $u(x)$, then
there is a developing map $h(x)$ such that $h(x)$ is a multi-valued meromorphic
function on $\C$ satisfying
\begin{equation*}
 u(x)=\ln\frac{8|h'(x)|^2}{(1+|h(x)|^2)^2},\qquad x\in
 \dot{\C}:=\mathbb{C}\setminus\{0,1,t_1,\dots,t_{m+n}\}.
\end{equation*}
This is known as the Liouville formula and this $h(x)$ is multi-valued
and has its projective monodromy group contained in $\PSU(2)$. More
precisely, given any $\ell\in \pi_1\bigl(\dot{\C}, q_0\bigr)$ where
$q_0\in\dot{\C}$ is base point, the analytic continuation of $h(x)$,
denoted by $\ell^*h(x)$, is also a developing map of~$u(x)$ and so
$\ell^*h(x)=\frac{ah(x)+b}{ch(x)+d}$ for some projective monodromy
matrix $\SM{a}{b}{c}{d}\in\PSU(2)$.

The curvature equation~\eqref{00eq: DE on C} has been extensively
studied in the past several decades. For the recent development of
this subject and its application, we refer
\cite{Chai-Lin-Wang,Chen-Lin-weight2,
 Chen-Lin-rectangular,Chen-Lin-Green-function,Chen-Lin-Yang,
Eremenko,Coaxial,EGT,Quadrilateral,Guo-Lin-Yang,
 Lin-Wang-bubbling,Luo-Tian,Mondello-Panov,Mondello-Panov2}
to the interested reader. So far, the existence of solutions for
\eqref{00eq: DE on C} with general singular sources is challenging
and seems to be out of reach.
Mondello and Panov \cite{Mondello-Panov} studied a reduced
problem: \emph{To describe possible angles $($i.e., to describe the data
 $\{\theta_0,
 \theta_1,\theta_\infty,\theta_{t_1},\dots,\theta_{t_{n+m}}\})$ such
 that~\eqref{00eq: DE on C} has solutions for some singular set
 $\{t_1,\dots,t_{n+m}\}$}.
They introduced the measure
$d_1\bigl(\Z_o^{n+m+3},{\boldsymbol\theta-\boldsymbol
 1}\bigr)=d_1\bigl(\Z_o^{n+m+3},{\boldsymbol\alpha}\bigr)$ on
\[
 \boldsymbol \theta=(\theta_0,\theta_1,\theta_\infty,\theta_{t_1},
 \dots,\theta_{t_{n+m}}),
\]
i.e.,
\[
 \boldsymbol\theta-\boldsymbol 1={\boldsymbol\alpha}
 =(\alpha_0,\alpha_1,\alpha_\infty,\alpha_{t_1},\dots,
 \alpha_{t_{n+m}}),
\]
where $\Z_o^{n+m+3}$ is the subset of $\Z^{n+m+3}$ consisting of
vectors with odd sums of coordinates, and~$d_1$~is the
$\ell_1$-distance. Then they proved the following interesting result.

\begin{Theorem}[{\cite{Mondello-Panov}}]\label{thm-MP}
 Suppose that the angles $\theta_0$, $\theta_1$, $\theta_\infty$ and
 $\theta_{t_i}$, $1\le i\le n+m$, are fixed. Then the following
 hold.
\begin{enumerate}\itemsep=0pt
\item[$1.$]
If~\eqref{00eq: DE on C} has a solution $u(x)$ for some singular set $\{t_1,\dots,t_{n+m}\}$, then
\[d_1\bigl(\Z_o^{n+m+3},{\boldsymbol\theta-\boldsymbol 1}\bigr)\geq 1.\]
Furthermore, if
 \begin{equation}\label{00e0q-5}
d_1\bigl(\Z_o^{n+m+3},{\boldsymbol\theta-\boldsymbol 1}\bigr)=1,
\end{equation}
then there is a developing map $h(x)$ such that the projective
monodromy group of $h(x)$ is contained in the unit circle, i.e., any
projective monodromy matrix of $h(x)$ is diagonal.

\item[$2.$]
Conversely, if $d_1\bigl(\Z_o^{n+m+3},{\boldsymbol\theta-\boldsymbol 1}\bigr)>1$, then there exists some singular set $\{t_1,\dots, t_{n+m}\}$ such that~\eqref{00eq: DE on C} has solutions.
\end{enumerate}
\end{Theorem}

\begin{Definition}
In \cite{Mondello-Panov}, the corresponding metric
$\frac12{\rm e}^{u(x)}|{\rm d}x|^2$ (if exists) is called \emph{co-axial} if there
is a developing map $h(x)$ such that the projective monodromy group of
$h(x)$ is contained in the unit circle. In this paper, we also call
this solution $u(x)$ \emph{co-axial} for convenience.
\end{Definition}

For the case~\eqref{00e0q-5}, Theorem~\ref{thm-MP} shows that solutions must be
co-axial if exist. However, the existence of co-axial metrics does not
necessarily imply~\eqref{00e0q-5}. Thus, there arises naturally the
following question: Are there sufficient conditions on angles that
guarantee the existence of a~singular set $\{t_1,\dots,t_{n+m}\}$ such that
\eqref{00eq: DE on C} has a co-axial solution? This question was
completely solved by Eremenko \cite{Coaxial}.

\begin{Theorem}[\cite{Coaxial}]\label{thm-Eremenko}
If~\eqref{00eq: DE on C} has a co-axial solution $u(x)$ for some singular set $\{t_1,\dots,t_{n+m}\}$, then
 there are $\epsilon_{p}\in\{\pm 1\}$ for $p\in
 \{0,1,\infty,t_1,\dots,t_n\}$ such that
\begin{gather}
k':=\sum_{p\in\{0,1,\infty,t_1,\dots,t_n\}} \epsilon_{p}\theta_{p}\in\mathbb{Z}_{\geq 0},\label{00e0qq-111}
\\
k'':=\sum_{j=n+1}^{n+m} \theta_{t_j}-(n+m+3)-k'+2\in 2\mathbb{Z}_{\geq 0}.\label{00e0qq-112}
\end{gather}
Moreover, if there is $\eta\in\mathbb{R}_{>0}$ such that
$\mathbf{c}=\eta\mathbf{b}=\eta(b_1,\dots,b_q)$ with
$b_j\in\mathbb{N}$, $\gcd(b_1,\dots,b_q)=1$, where
\[
 \boldsymbol c:=(\theta_0,\theta_1,\theta_\infty,\theta_{t_1},
 \dots,\theta_{t_n},\underbrace{1,\dots,1}_{k'+k''}),
\]
then we also have
\begin{equation}\label{k3}
 2\max_{n+1\leq j\leq {n+m}}\theta_{t_j}\leq \sum_{j=1}^q b_j.
\end{equation}

Conversely, if there is no such $\eta$ such that
$\mathbf{c}=\eta\mathbf{b}$, then
\eqref{00e0qq-111} and~\eqref{00e0qq-112} are also sufficient; if there
is such $\eta$ such that $\mathbf{c}=\eta\mathbf{b}$, then
\eqref{00e0qq-111}--\eqref{k3} are also sufficient.
\end{Theorem}

As we have remarked earlier, the existence of co-axial solutions does
not necessarily imply~\eqref{00e0q-5}. Let us recall Eremenko's
example \cite{Coaxial}:
$(\theta_0,\theta_1,\theta_\infty,\theta_{t_1},\theta_{t_2})=(\beta,\beta,2\beta,2\beta,3)$
for some $\beta\in\mathbb{R}_{>0}\setminus\N$, for which it follows
from Theorem~\ref{thm-Eremenko} that co-axial solutions exist but it
does not satisfy~\eqref{00e0q-5}.

Here we have an interesting observation about Mondello--Panov's
condition~\eqref{00e0q-5} and Eremenko's condition~\eqref{k3}.

\begin{Theorem}\label{thm-MP-E}
Suppose that~\eqref{00e0q-5} holds. Then
\eqref{00e0qq-111} and~\eqref{00e0qq-112} imply~\eqref{k3} automatically,
namely~\eqref{00eq: DE on C} has $($co-axial$)$ solutions for some
singular set $\{t_1,\dots,t_{n+m}\}$ if and only if
\eqref{00e0qq-111} and~\eqref{00e0qq-112} hold for some
$\epsilon_p\in\{\pm 1\}$.
\end{Theorem}

Theorem~\ref{thm-MP-E} will be proved in Section~\ref{sec-Eremenko}.
Define
\[
 \triangle_{n+m}:=\bigl\{(t_1,\dots,t_{n+m})\in \mathbb{C}^{n+m}\mid
 t_j=0,1 \text{ or } t_j=t_k\;\text{for some $j\neq k$}\bigr\}.
\]
Throughout the paper, we let $\theta_0$, $\theta_1$,
$\theta_\infty$, and $\theta_j'$, $j=1,\dots,n+m$, be fixed real
numbers such that $\theta_0$, $\theta_1$, $\theta_\infty$, and
$\theta_j'$, $j=1,\dots,n$, are non-integers and $\theta_j'$,
$j=n+1,\dots,n+m$, are integers greater than $1$.
Let
\begin{gather}
 A=A_\theta:=\bigl\{(t_1,\dots,t_{n+m})\in \mathbb{C}^{n+m}\setminus
 \triangle_{n+m}\mid \text{\eqref{00eq: DE on C}
 with}\ \theta_{t_j}=\theta_j',
\nonumber\\ \hphantom{ A=A_\theta:=\{}{}
j=1,\dots,n+m, \
 \text{has co-axial solutions}\bigr\}.\label{eq: set A}
\end{gather}
An $(n+m)$-tuple $(t_1,\dots,t_{n+m})$ in $A$ will be called
admissible for~\eqref{00eq: DE on C}.
By Theorem~\ref{thm-Eremenko}, we know that $A\neq\varnothing$ if and
only if~\eqref{00e0qq-111}--\eqref{k3} hold. In this case, a natural
question that interest us is:
\[\parbox{\dimexpr\linewidth-5em}{\it
 What are the geometric and algebaric properties of the set $A$, such as its dimension and so on?
 }\]
This paper aims to study this problem. Note that~\eqref{00e0qq-112}
implies $m\geq 1$, this is the reason why we assume $m\geq 1$ in this paper.

Let us consider the special case $m=1$ first. Denote
\[
 \Q_{\theta}:=\Q\bigl(\theta_0,\theta_1,\theta_\infty,\theta_1',
 \dots,\theta_n'\bigr)\subsetneqq \R.
\]

\begin{Theorem}\label{00thm-2-2}
Let $m=1$ and~\eqref{00e0qq-111}--\eqref{k3} hold with
$\theta_{t_j}=\theta_j'$, i.e., $A\neq\varnothing$.
Then $A$ is a finite set with $\# A\leq
2^{n+2}\times\bigl(\theta_{n+1}'-1\bigr)!$. Furthermore, for each
$\boldsymbol t=(t_1,\dots,t_{n+1})\in A$, all $t_j$'s are algebraic
over $\Q_{\theta}$ and the field
\[\Q_\theta(A):=\Q_{\theta}\bigl(\bigl\{t_j \mid (t_1,\dots,t_{n+1})\in A\bigr\}\bigr)\]
is a Galois extension of $\Q_{\theta}$ and
\[[\Q_\theta(A) : \Q_\theta]\leq M<\infty,\]
where $M$ is a constant depending on $n$ and $\theta_{n+1}'$.
\end{Theorem}

In the special case when there are only three
non-integer angles, i.e., when $n=0$, we can get a much sharper bound
for the cardinality of $A$. In the following, we will write
$t_1$ and $\theta_1'$ simply by $t$ and $\theta$.

\begin{Theorem} \label{theorem: sharp}
 Assume that $n=0$ and $m=1$. Assume that
 $\epsilon_0,\epsilon_1\in\{\pm1\}$ are signs such that
 $k:=\theta_\infty+\epsilon_0\theta_0+\epsilon_1\theta_1\in\Z$ and
 $\theta\equiv k\mod 2$. Then the cardinality of the set $A$ is
 \[
 \begin{cases}
 0, &\text{if }\theta\le|k|, \\
 \bigl(\theta^2-k^2\bigr)/4, &\text{if }\theta>|k|, \end{cases}
 \]
 counted with multiplicities $($see Section~$\ref{section: proof of
 Theorem 1.4}$ for an explanation of
 multiplicities$)$. Moreover, if ${t\in A}$, then the degree of $t$ over
 $\Q_\theta$ is at most $\bigl(\theta^2-k^2\bigr)/4$.
\end{Theorem}

We remark that the condition $\theta\equiv k\mod 2$ is necessary for
$A\neq\varnothing$, by~\eqref{00e0qq-112}. Interestingly, a key
ingredient in the proof of the theorem is the monodromy of the
classical hypergeometric equations.

\begin{Remark} Theorem~\ref{theorem: sharp} yields a sharp non-existence
result for~\eqref{00eq: DE on C} when $n=0$ and $m=1$. That is,
when $\theta\equiv k\mod 2$, if $t\in\C$ is transcendental over
$\Q_\theta$ or is algebraic of degree greater than $\bigl(\theta^2-k^2\bigr)/4$
over $\Q_\theta$, then~\eqref{00eq: DE on C} has no solutions (not just co-axial solutions). Indeed, suppose~\eqref{00eq: DE on C} has a solution $u(z)$, since $\theta\equiv k\mod 2$ easily yields $d_1\bigl(\Z_o^{4},{\boldsymbol\theta-\boldsymbol 1}\bigr)=1,$ it follows from Theorem~\ref{thm-MP} that $u(z)$ is co-axial, i.e., $t\in A$, a contradiction with Theorem~\ref{theorem: sharp}.
 It would be an interesting problem to obtain a similar result for the
case $n>0$.\end{Remark}

Let $\theta=2$ in Theorem~\ref{theorem: sharp}. Then the set $A$ is non-empty
 if and only if $k=0$, and in this case, we will prove in Example~\ref{example1} that $A=\{-\epsilon_0\theta_0/\theta_\infty\}$, so we obtain the following corollary.

\begin{Corollary}[see Example~\ref{example1}]
Consider the curvature equation
\begin{gather}
 \Delta u+{\rm e}^u=4\pi((\theta_0-1)\delta_0+(\theta_1-1)\delta_1
 +\delta_{t})
 \qquad\text{in}\ \mathbb R^2\cong\C, \nonumber\\
 u(x)=-2(1+\theta_\infty)\ln|x|+O(1) \qquad \text{as} \
 |x|\to\infty, \label{Liou-eq2}
\end{gather}
with $\theta_0,\theta_1,\theta_\infty\in\mathbb{R}_{>0}\setminus \mathbb{N}$. Then~\eqref{Liou-eq2} has co-axial solutions for some $t\in\mathbb{C}\setminus\{0,1\}$ if and only if there are $\epsilon_0,\epsilon_1\in\{\pm1\}$ such that
 $\theta_\infty+\epsilon_0\theta_0+\epsilon_1\theta_1=0$. In this case,
 \begin{itemize}\itemsep=0pt
 \item[$1)$]if $t\neq -\epsilon_0\theta_0/\theta_\infty$, then~\eqref{Liou-eq2} has no solutions;
 \item[$2)$]if $t= -\epsilon_0\theta_0/\theta_\infty$, then~\eqref{Liou-eq2} has co-axial solutions.
 \end{itemize}
\end{Corollary}

Let $\theta=3$ in Theorem~\ref{theorem: sharp}. Then the set $A$ is non-empty
 only when $|k|=1$. In this case,
 to simplify notations, we will write $\epsilon_0\theta_0$ and
 $\epsilon_1\theta_1$ simply by $\theta_0$ and $\theta_1$. Then we will prove in Example~\ref{example2} that $A=\Bigl\{\frac{\theta_0\theta_\infty
 \pm\sqrt{-k\theta_0\theta_1\theta_\infty}}
 {(\theta_0+\theta_1)\theta_\infty}\Bigr\}$, so the following result holds.

 \begin{Corollary}[see Example~\ref{example2}]
Consider the curvature equation
\begin{gather}
 \Delta u+{\rm e}^u=4\pi((\theta_0-1)\delta_0+(\theta_1-1)\delta_1
 +2\delta_{t})
 \qquad\text{in} \ \mathbb R^2\cong\C, \nonumber\\
 u(x)=-2(1+\theta_\infty)\ln|x|+O(1) \qquad \text{as} \
 |x|\to\infty, \label{Liou-eq3}
\end{gather}
with $\theta_0,\theta_1,\theta_\infty\in\mathbb{R}_{>0}\setminus \mathbb{N}$. Then~\eqref{Liou-eq3} has co-axial solutions for some $t\in\mathbb{C}\setminus\{0,1\}$ if and only if there are $\epsilon_0,\epsilon_1\in\{\pm1\}$ such that
 $k:=\theta_\infty+\epsilon_0\theta_0+\epsilon_1\theta_1=\pm 1$. In this case $($we write~$\epsilon_0\theta_0$ and~$\epsilon_1\theta_1$
 simply by $\theta_0$ and $\theta_1$ for convenience$)$,
 \begin{itemize}\itemsep=0pt
 \item[$1)$]if $t\notin\Big\{\frac{\theta_0\theta_\infty
 \pm\sqrt{-k\theta_0\theta_1\theta_\infty}}
 {(\theta_0+\theta_1)\theta_\infty}\Big\}$, then~\eqref{Liou-eq3} has no solutions;
 \item[$2)$]if $t\in\Big\{ \frac{\theta_0\theta_\infty
 \pm\sqrt{-k\theta_0\theta_1\theta_\infty}}
 {(\theta_0+\theta_1)\theta_\infty}\Big\}$, then~\eqref{Liou-eq3} has co-axial solutions.
 \end{itemize}
\end{Corollary}

For general cases $m\geq 2$, we can not expect that $A$ is a finite
set. Let $\overline{\mathbb Q}$ denote the algebraic closure of $\mathbb{Q}$.

\begin{Theorem}\label{00thm-2-1}
Let $m\geq 2$ and~\eqref{00e0qq-111}--\eqref{k3} hold,
i.e., $A\neq\varnothing$. Then $A$ is an algebraic set of dimension $\leq
m-1$. Furthermore, if we further assume
\begin{equation} \label{000eq: theta alpha a}
 \theta_0,\theta_1,\theta_\infty,\theta_1',\dots,\theta_n'
 \in \overline{\Q}_{>0}\setminus\mathbb{Z},
\end{equation}
then for any $\boldsymbol t=(t_1,\dots,t_{n+m})\in A$, the transcendence degree of
 $\overline{\Q}(\boldsymbol t)$ over $\overline{\Q}$ is $\le m-1$.
\end{Theorem}

One of the main ideas in our proofs of all the main results is the
integrability of the curvature equation. This integrability property
allows us to connect the partial differential equation with a
second-order Fuchsian ordinary differential equation. In Section
\ref{sec-MP-E}, we will derive those ODE's. Applying integrability in
terms of the developing maps, we will prove an equivalent condition
for the existence of co-axial solutions. Based on this equivalent
condition, we will prove Theorems~\ref{00thm-2-2} and~\ref{00thm-2-1}
in Section~\ref{sec-dimension}. In Section~\ref{sec-Eremenko}, we give
a proof of Theorem~\ref{thm-MP-E}. In the final section, we present a
proof of Theorem~\ref{theorem: sharp} using hypergeometric equations.

\section{The integrability of curvature equations and complex ODEs}

\label{sec-MP-E}

In this section, we briefly review how to connect the integrability of the curvature equation~\eqref{00eq: DE on C} with a class of second order ODEs in complex variables; see, e.g., \cite{Chai-Lin-Wang,Coaxial}. Given a solution $u(x)$ of~\eqref{00eq: DE on C}, we define
\begin{equation*}
 Q(x):=-\frac12\biggl(u_{xx}(x)-\frac12 u_x(x)^2\biggr),\qquad
 x\in\mathbb{C}.
\end{equation*}
Then it is easy to see $Q_{\bar{x}}(x)\equiv 0$, so $Q(x)$ is a meromorphic function in $\mathbb{C}$ with poles belonging to
\[
 \mathcal{I}:=\{0,1,t_1,\dots,t_{n+m}\}.
\]
More precisely, at each $p\in\mathcal{I}$, since $\Delta (\ln|x-p|)=2\pi\delta_p$, by~\eqref{00eq: DE on C} and a standard elliptic regularity argument (cf. \cite{BM}), we obtain that $u(x)-2\alpha_p\ln|x-p|$ is bounded near $p$, i.e.,
\[u(x)=2\alpha_p\ln|x-p|+O(1)\qquad\text{for $x$ near $p$}.\]
Consequently,
\[Q(x)=\frac{\alpha_p(\alpha_p+2)}{4}(x-p)^{-2}+O\bigl((x-p)^{-1}\bigr)\qquad\text{for $x$ near $p$},\]
and
\[Q(x)=\frac{\alpha_\infty(\alpha_\infty+2)}{4}x^{-2}+O\bigl(x^{-3}\bigr)\qquad\text{for $x\to \infty$}.\]
Thus,
\begin{align}\label{e0q-30}
Q(x)=\sum_{p\in\mathcal{I}}\biggl[\frac{\beta_p}{(x-p)^2}+\frac{d_p}{(x-p)}\biggr],
\end{align}
where
\begin{equation}\label{e0q-31}
 \beta_p:=\frac{\alpha_{p}(\alpha_{p}+2)}{4},\qquad p\in \{0,1,t_1,\dots, t_{n+m},\infty\},
\end{equation}
and $d_p$'s are some constants satisfying
\begin{align}\label{e0q-32}
\sum_{p\in \mathcal{I}}d_p=0,\qquad \sum_{p\in \mathcal{I}}(\beta_p+pd_p)=\beta_\infty.
\end{align}
Thus, $d_0$ and $d_1$ can be uniquely determined by
\[
 \boldsymbol d:=(d_{t_1},\dots, d_{t_{n+m}}),\qquad
 \bt:=(t_1,\dots,t_{n+m})
\]
as
\begin{align*}
d_1=\beta_\infty-\sum_{p\in\mathcal{I}\setminus\{0,1\}}(\beta_p+pd_p),
 \qquad d_0=\sum_{p\in \mathcal{I}\setminus\{0,1\}}
 (\beta_p+(p-1)d_p)-\beta_\infty,
\end{align*}
 while $\boldsymbol d$ could be considered as $n+m$ unknown coefficients.

On the other hand, as mentioned in the introduction, the Liouville
theorem (see, e.g., \cite{Chai-Lin-Wang}) says that for a solution
$u(x)$ of~\eqref{00eq: DE on C}, there is a developing map $h(x)$ such
that
\begin{equation}\label{eqLivoulle-c1}
u(x)=\ln\frac{8|h'(x)|^2}{(1+|h(x)|^2)^2},\qquad x\in \dot{\C}:=\mathbb{C}\setminus\mathcal{I}.
\end{equation}
The developing map $h(x)$ is multi-valued in $\mathbb{C}$ and due to
\eqref{eqLivoulle-c1}, \emph{the projective monodromy group of $h(x)$
 is contained in $\PSU(2)$}. More precisely, given any $\ell\in
\pi_1\bigl(\dot{\C}, q_0\bigr)$ where $q_0\in\dot{\C}$ is a base point, the
analytic continuation of $h(x)$, denoted by $\ell^*h(x)$, is also a
developing map of $u(x)$ and so $\ell^*h(x)=\frac{ah(x)+b}{ch(x)+d}$
for some projective monodromy matrix $\SM{a}{b}{c}{d}\in\PSU(2)$.

To connect $Q(x)$ and $h(x)$, we note that~\eqref{eqLivoulle-c1} implies
\begin{equation}\label{e22q-1}
\{h(x),x\}:=\biggl(\frac{h''(x)}{h'(x)}\biggr)'-\frac12 \biggl(\frac{h''(x)}{h'(x)}\biggr)^2=u_{xx}(x)-\frac12 u_x(x)^2=-2Q(x),
\end{equation}
where $\{h(x),x\}$ is the Schwarz derivative. We refer the reader to
\cite{Chai-Lin-Wang} for details of computation.

Consider the following second order linear ODE with the complex variable $x$:
\begin{equation}\label{00eq: ODE on C}
y''(x)=Q(x)y(x),\qquad x\in\mathbb{C}.
\end{equation}
Then a classical result (see \cite{Whittaker-Watson}) says that for
any basis $(\hat{y}_1,\hat{y}_2)$ of~\eqref{00eq: ODE on C}, the
Schwarz derivative $\{\hat{y}_2/\hat{y}_1,x\}$ always satisfies
\begin{equation*}
\{\hat{y}_2/\hat{y}_1,x\}=-2Q(x).
\end{equation*}
From this and~\eqref{e22q-1}, we conclude that if~\eqref{00eq: ODE on
 C} is derived from a solution of~\eqref{00eq: DE on C}, then there
is a basis $(y_1,y_2)$ of~\eqref{00eq: ODE on C} such that
\begin{equation}\label{e22q-2}
h(x)=y_2(x)/y_1(x).
\end{equation}
Furthermore, the projective monodromy group of~\eqref{00eq: ODE on C}
with respect to $y_1$ and $y_2$ is the same as that of $h(x)$, i.e., is
contained in $\PSU(2)$.

Conversely, given $Q(x)$ via~\eqref{e0q-30}--\eqref{e0q-32}, if the
corresponding ODE~\eqref{00eq: ODE on C} has a basis $(y_1,y_2)$ such
that the projective monodromy group of the ratio $h(x)=y_2(x)/y_1(x)$
is contained in $\PSU(2)$, then
$u(x):=\ln\frac{8|h'(x)|^2}{(1+|h(x)|^2)^2}$ is a solution of
\eqref{00eq: DE on C}.

For the reader's convenience, we now briefly recall some basic notions
of ODEs like~\eqref{00eq: ODE on C}. We refer the reader to
\cite{Yoshida,Whittaker-Watson} for a comprehensive
introduction. Generally,~\eqref{00eq: ODE on C} is called
\emph{Fuchsian} because the pole order of $Q(x)$ is at most $2$ at any
singularities.
Note that the Riemann scheme (see \cite[p.~11]{Yoshida} for the definition of the Riemann scheme) of~\eqref{00eq: ODE on C} is
 given by
\begin{equation}\label{00e0q-10}\begin{pmatrix}
0&1&t_1&\cdots& t_{n+m}&\infty\\
-\frac{{\alpha}_0}{2} &-\frac{{\alpha}_1}{2}&-\frac{{\alpha}_{t_1}}{2}&\cdots&-\frac{{\alpha}_{t_{n+m}}}{2}&-(\frac{{\alpha}_\infty}{2}+1)\\
\frac{{\alpha}_0}{2}+1&\frac{{\alpha}_1}{2}+1&\frac{{\alpha}_{t_1}}{2}+1&\cdots&\frac{{\alpha}_{t_{n+m}}}{2}+1&\frac{{\alpha}_\infty}{2}
\end{pmatrix}.\end{equation}
Observe that the exponent difference at $t_j$ is an integer for
$n+1\leq j\leq n+m$.
Since $Q(x)$ comes from a solution $u(x)$ of (\ref{00eq: DE on C}), it is well known that
(\ref{00eq: ODE on C}) is apparent at $t_j$ \big(i.e., no logarithmic
singularity near $t_j$\big) and so
the local monodromy matrix of~\eqref{00eq: ODE on C} at $t_j$ is
$(-1)^{{\alpha}_{t_j}}I_2$ for all $n+1\leq j\leq n+m$.
In this paper, we will also consider a Fuchsian differential equation
\eqref{00eq: ODE on C} which may not come from a solution of
\eqref{00eq: DE on C}. In this case, unless otherwise specified,
\emph{we always assume that~\eqref{00eq: ODE on C} with the Riemann
 scheme~\eqref{00e0q-10} is apparent at $t_j$ for all $n+1\leq j\leq
 n+m$}.

It is well known that the necessary and sufficient condition of
\eqref{00eq: ODE on C} being apparent at $t_j$ for all $n+1\leq j\leq
n+m$ can be derived from the Frobenius method.

\begin{Theorem} \label{prop: apparent}
 There are polynomials
 \[
 \hat{\mathcal{P}}_j(\boldsymbol d,\boldsymbol t)\in\Q_\theta(\bt)[\boldsymbol d]=d_{t_{j}}^{\theta_{t_j}}+ \text{l.o.t.\ in $\boldsymbol d$} \in \Q_\theta(\boldsymbol t)[\boldsymbol d]
 \]
 of total degree $\theta_{t_j}$ in $\boldsymbol d$, $j=n+1,\dots,n+m$,
 such that~\eqref{00eq: ODE on C}
 is apparent at $t_j$ for all $n+1\leq j\leq n+m$ if
 and only if $\boldsymbol{d}$ and $\bt$ satisfy $\hat{P}_j(\boldsymbol{d},\bt)=0$ for all
 $j=n+1,\dots,n+m$.

 Furthermore,
 $\hat{\mathcal{P}}_j(\boldsymbol{d},\bt)$ can be written as
 \begin{equation}\label{e22q:PP}\hat{\mathcal{P}}_j(\boldsymbol{d},\bt)=\frac{\mathcal{P}_j(\boldsymbol{d},\bt)}
 {\prod_{p\in\mathcal{I}\setminus\{t_j\}}(p-t_j)^{N_j}}\end{equation}
 for some $N_j\in\mathbb{Z}_{\geq 0}$ and $\mathcal{P}_j(\boldsymbol{d},\bt)\in \Q_{\theta}[\boldsymbol{d},\bt]$ such that $\mathcal{P}_j(\boldsymbol{d},\bt)$ and $\prod_{p\in\mathcal{I}\setminus\{t_j\}}(p-t_j)^{N_j}$ are coprime. Thus,~\eqref{00eq: ODE on C}
 is apparent at $t_j$ for all $n+1\leq j\leq n+m$ if
 and only if $\mathcal{P}_j(\boldsymbol{d},\bt)=0$ for all
 $j=n+1,\dots,n+m$.
\end{Theorem}

\begin{proof} Let us take $t_{n+m}$ for example and in the following proof, we write $t_{n+m}=t$ for convenience. By the Frobenius method, (\ref{00eq: ODE on C}) is apparent at the singularity $t$ if and only if it has a~solution of the form
\[y(x)=\sum_{j=0}^{\infty}c_j u^{j-\alpha_{t}/2},\qquad u=x-t,\quad c_0=1.\]
Observe that
{\allowdisplaybreaks
\begin{align*}
Q(x)&=\sum_{p\in\mathcal{I}}\biggl[\frac{\beta_p}{(x-p)^2}+\frac{d_p}{(x-p)}\biggr]
=
\frac{\beta_t}{u^2}+\frac{d_t}{u}
+\sum_{p\in\mathcal{I}\setminus\{t\}}
\biggl[\frac{\beta_p}{(u+t-p)^2}+\frac{d_p}{u+t-p}\biggr]\\
&=\frac{\beta_t}{u^2}+\frac{d_t}{u}
+\sum_{j=0}^\infty(-1)^{j} \sum_{p\in\mathcal{I}\setminus\{t\}}\frac{\beta_p(j+1)+d_p(t-p)}{(t-p)^{j+2}}u^j
=\sum_{j=0}^{\infty}\eta_j u^{j-2},
\end{align*}
}%
where
\[\eta_0=\beta_t,\qquad \eta_1=d_t,\qquad
 \eta_j=(-1)^{j} \sum_{p\in\mathcal{I}\setminus\{t\}}\frac{\beta_p(j-1)+d_p(t-p)}{(t-p)^{j}},\qquad j\geq 2,
 \]
i.e., $\prod_{p\in\mathcal{I}\setminus\{t\}}(t-p)^j\eta_j\in \Q_\theta[\boldsymbol d,\boldsymbol t]$.

Consequently, $y''(x)=Q(x)y(x)$ if and only if
\begin{equation}\label{e0q-14}
j(j-\theta_t)c_j=\sum_{k=0}^{j-1}\eta_{j-k}c_k,\qquad\forall j\geq 0,
\end{equation}
where $(j-\alpha_t/2)(j-\alpha_t/2-1)-\eta_0=j(j-\alpha_t-1)=j(j-\theta_t)$ is used. Clearly, (\ref{e0q-14}) holds automatically for $j=0$, and (\ref{e0q-14}) with $j=1$ leads to $c_1=\frac{d_t}{1-\theta_t}$. By an induction argument, for any $2\leq j\leq \theta_t-1$, $c_j$ can be uniquely solved by (\ref{e0q-14}) as
\[c_j=r_j\bigl[d_{t}^j+\text{l.o.t. in $\boldsymbol d$}\bigr]\in \Q_\theta(\boldsymbol t)[\boldsymbol d]\]
with total degree $j$ in $\boldsymbol d$,
where $r_j\in\Q_\theta\setminus\{0\}$. Furthermore,
\[\prod_{p\in\mathcal{I}\setminus\{t\}}(t-p)^j c_j\in \Q_\theta[\boldsymbol d,\boldsymbol t].\]
Consequently, (\ref{e0q-14}) with $j=\theta_t$ leads to the existence of $r_{\theta_t}\in\Q_\theta\setminus\{0\}$ and a polynomial
\[\hat{P}(\boldsymbol d, \boldsymbol t)=d_t^{\theta_t}+\text{l.o.t.\ in $\boldsymbol d$}\in \Q_\theta(\boldsymbol t)[\boldsymbol d]\]
with total degree $\theta_t$ in $\boldsymbol d$ such that
\[
 \sum_{k=0}^{\theta_t-1}\eta_{j-k}c_k=r_{\theta_t}
 \hat{P}(\boldsymbol d, \boldsymbol t).
\]
Then it is standard by the Frobenius theory that (\ref{00eq: ODE on
 C}) is apparent at the singularity $t$ if and only if
$\hat{P}(\boldsymbol d, \boldsymbol t)=0$.
Furthermore, the above argument also implies the existence of an integer $0\leq N\leq \theta_t$ such that
\[
 P(\boldsymbol d, \boldsymbol t):=\hat{P}(\boldsymbol d, \boldsymbol
 t)\prod_{p\in\mathcal{I}\setminus\{t\}}(t-p)^{N}\in
 \Q_\theta[\boldsymbol d,\boldsymbol t]
\]
and $P(\boldsymbol d, \boldsymbol t)$, $\prod_{p\in\mathcal{I}\setminus\{t\}}(t-p)^{N}$ are coprime.
\end{proof}

The case $n=0$ is special and we plan to study the solvability of
\eqref{00eq: DE on C} for this case in another paper. In this case,
the ODE~\eqref{00eq: ODE on C} can be transformed into a Heun equation
and this Heun equation has been well studied in
\cite{Eremenko,Eremenko-Fuchsian}. In Section~\ref{section: proof of
 Theorem 1.4}, we will use some results from~\cite{Eremenko} and
hypergeometric equations to prove Theorem~\ref{theorem: sharp}.

\section{Proofs of Theorems~\ref{00thm-2-2} and~\ref{00thm-2-1}}\label{sec-dimension}

This section is devoted to the proofs of Theorems
\ref{00thm-2-2} and~\ref{00thm-2-1}. In the following we use notations
\[
 \mathcal{I}:=\{0,1,t_1,\dots,t_{n+m}\},\qquad
 \mathcal{I}_1:=\{0,1,t_1,\dots,t_{n}\}.
\]

\subsection{General setting}

As mentioned in the introduction, following \cite{Mondello-Panov}, we
call the metric $\frac12{\rm e}^{u(x)}|{\rm d}x|^2$ and also the solution $u(x)$
of~\eqref{00eq: DE on C} to be \emph{co-axial} if there is a
developing map $h(x)$ of $u(x)$ such that the projective monodromy
group of $h(x)$ is contained in the unit circle, that is, for any
$\gamma\in \pi_1\bigl(\dot{\C}, q_0\bigr)$, the analytic continuation of $h(x)$
along $\gamma$, denoted by $\gamma^*h(x)$, satisfies
$\gamma^*h(x)=\lambda(\gamma)h(x)$ for some
$\lambda(\gamma)\in\mathbb{C}$ satisfying
$|\lambda(\gamma)|=1$. Clearly, this is equivalent to that the
monodromy group of the associated ODE~\eqref{00eq: ODE on C} with
$Q(x)=-\frac12\bigl(u_{xx}(x)-\frac12 u_x(x)^2\bigr)$ is \emph{commutative}. In
Section~\ref{sec-MP-E}, we have discussed the equivalence between the
existence of a solution $u(x)$ of~\eqref{00eq: DE on C}
and a~Fuchsian ODE~\eqref{00eq: ODE on C} with $Q(x)$ given by
\eqref{e0q-30} satisfying $\mathcal{P}_j(\boldsymbol{d},\bt)=0$ for
all $n+1\leq j\leq n+m$ (see Theorem~\ref{prop: apparent}) such that
the monodromy group is conjugate to a subgroup of
$\mathrm{SU}(2)$. This equivalence can be further strengthened as
follows when the solution $u(x)$ is co-axial.

In fact, given such an ODE~\eqref{00eq: ODE on C}, take the basis of local solutions near $\infty$ as follows
\begin{align}
&y_{+}(x)=x^{-\frac{\alpha_\infty}{2}}\sum_{j=0}^\infty c_{+,j}x^{-j},\qquad c_{+,0}=1, \nonumber\\
&y_{-}(x)=x^{\frac{\alpha_\infty}{2}+1}\sum_{j=0}^\infty
c_{-,j}x^{-j},\qquad c_{-,0}=1. \label{00e0q-y+}
\end{align}
Suppose that~\eqref{00eq: ODE on C} comes from a co-axial metric.
By~\eqref{e22q-2}, we have $h(x)=y_2(x)/y_1(x)$ for some solutions
$y_i(x)$ of~\eqref{00eq: ODE on C}. Consider a large circle $|x|=R$
and let $h\bigl({\rm e}^{2\pi {\rm i}}x\bigr)$ be the analytic continuation of $h(x)$ along
the circle. Then by the co-axial condition, one has
\[
h\bigl({\rm e}^{2\pi {\rm i}}x\bigr)=\lambda h(x)
\]
for some $\lambda\neq 0$. On the other hand, there is a matrix
$M_\infty=\SM abcd\in\SL(2,\C)$ such that
\[
\begin{pmatrix}
 y_1\bigl({\rm e}^{2\pi {\rm i}}x\bigr) \\y_2\bigl({\rm e}^{2\pi {\rm i}}x\bigr) \end{pmatrix}
=M_\infty\begin{pmatrix}y_1(x)\\y_2(x)\end{pmatrix}.
\]
These two identities together imply that $b=c=0$. Thus, after
multiplying by some nonzero constants, we may assume that
\begin{equation}\label{200h}
 h(x):=y_{+}(x)/y_{-}(x),
\end{equation}
which satisfies
\begin{equation}
 \label{h-infty}h(x)=x^{-1-\alpha_\infty}\bigl(1+O\bigl(x^{-1}\bigr)\bigr)=x^{-\theta_\infty}\bigl(1+O\bigl(x^{-1}\bigr)\bigr)\qquad
 \text{as}\ x\to\infty.
\end{equation}
Then under the basis $(y_{+}(x), y_{-}(x))$, we have
$M_{\infty}=\SM{{\rm e}^{-\pi {\rm i} \alpha_\infty}}{0}{0}{{\rm e}^{\pi {\rm i}
 \alpha_\infty}}$ with ${\rm e}^{-\pi {\rm i} \alpha_\infty}\neq {\rm e}^{\pi {\rm i}
 \alpha_\infty}$. Here for any $p\in \mathcal{I}\cup\{\infty\}$, we
use $M_p\in\SL(2,\mathbb{C})$ to denote the monodromy matrix of ODE
\eqref{00eq: ODE on C} with respect to a simple loop in
$\pi_1\bigl(\dot{\C}, q_0\bigr)$ that encircles $p$ once.
Now suppose that the monodromy group of~\eqref{00eq: ODE on C} is commutative (this holds if~\eqref{00eq: ODE on C} comes from a co-axial solution $u(x)$), then it follows that all the monodromy matrices under the basis $(y_{+}(x), y_{-}(x))$ are diagonal, i.e., $M_{p}=\SM{{\rm e}^{\pm\pi {\rm i} \alpha_{p}}}{0}{0}{{\rm e}^{\mp\pi {\rm i} \alpha_{p}}}$ for $p\in \mathcal{I}_1=\{0,1,t_1,\dots,t_n\}$. This implies that
 after analytic continuation,
 \[y_{\pm}(x)=(x-p)^{\rho_p^{\pm}}\bigl(c_p^{\pm}+O(|x-p|)\bigr)\qquad \text{near }p\in \mathcal{I}_1,\]
 where
 $\{\rho_p^{+},\rho_p^{-}\}=\bigl\{-\frac{\alpha_p}{2},\frac{\alpha_p}{2}+1\bigr\}$
 are the local exponents of~\eqref{00eq: ODE on C} at $p$ and
 $c_p^{\pm}\neq 0$.
Thus there exists $\epsilon_p\in\{\pm 1\}$ for $p\in \mathcal{I}_1$ such that
\[h(x)=(x-p)^{\epsilon_p\theta_p}(c_p+O(x-p))\qquad \text{as }x\to p\in \mathcal{I}_1,\quad c_p\neq 0.\]
From here and~\eqref{h-infty}, we see that $h(x)$ can be written as
\begin{equation}\label{00e0q-080}
 h(x)=\hat{h}(x)\prod_{p\in \mathcal{I}_1}(x-p)^{\epsilon_{p}\theta_{p}},
\end{equation}
where $\hat{h}(x)$ is \emph{meromorphic} in $\mathbb{C}$ and satisfies $\hat{h}(x)=x^{-\theta_\infty-\epsilon_0\theta_0-\epsilon_1\theta_1-\sum_{j=1}^n\epsilon_{t_j}\theta_{t_j}}\bigl(1+O\bigl(x^{-1}\bigr)\bigr)$ as $x\to\infty$, so
\begin{equation}\label{hhat}\emph{$\hat{h}(x)$ is a rational function}.\end{equation}

Conversely, given an ODE~\eqref{00eq: ODE on C} with~\eqref{00e0q-10}
which might not come from a solution of~\eqref{00eq: DE on C}, we
consider $h(x)=y_{+}(x)/y_{-}(x)$ as in
\eqref{00e0q-y+} and~\eqref{200h}. If there are $\epsilon_p\in\{\pm 1\}$
for $p\in \mathcal{I}_1$ such that~$h(x)$ has the expression
\eqref{00e0q-080} with $\hat h(x)$ being a rational function, then for
any $\gamma\in\pi_1\bigl(\dot{\mathbb{C}},q_0\bigr)$,
$\gamma^*h(x)=\lambda(\gamma)h(x)$ with $|\lambda(\gamma)|=1$
(i.e., the projective monodromy group of $h(x)$ is a subgroup of the
unit circle, or equivalently the monodromy matrices of (\ref{00eq: ODE
 on C}) are all diagonal under the basis $(y_{+}(x),
y_{-}(x))$). Thus
\[u(x):=\ln\frac{8|h'(x)|^2}{(1+|h(x)|^2)^2}\]
is well defined in $\mathbb{C}$. Then a direct computation shows that $u(x)$ is a co-axial solution of~\eqref{00eq: DE on C} with $h(x)$ being a developing map. Therefore, the above argument proves the following result.

\begin{Lemma}\label{lemm3-1}
 The equation~\eqref{00eq: DE on C} has co-axial solutions if and
 only if there is an ODE~\eqref{00eq: ODE on C} with the Riemann
 scheme~\eqref{00e0q-10} such that
 $h(x)=y_{+}(x)/y_{-}(x)$ has the expression~\eqref{00e0q-080} for
 some $\epsilon_p\in\{\pm 1\}$ for $p\in \mathcal{I}_1$ and some
 rational function $\hat{h}(x)$.
\end{Lemma}

Remark that $h(x)$ or equivalently $\hat{h}(x)$ might have zeros or
poles at $t_{n+j}$ for some $j\geq 1$. More precisely, the Riemann
scheme (\ref{00e0q-10}) indicates that the exponent difference is
$\alpha_{t_j}+1=\theta_{t_j}$ at $t_j$, so near $t_j$ with $n+1\leq
j\leq n+m$, we have
\[h(x)=\frac{y_{+}(x)}{y_{-}(x)}=\bigl(x-t_j\bigr)^{\beta}\bigl(c_{t_j}+O\bigl(x-t_j\bigr)\bigr) \qquad \text{for some $\beta\in\bigl\{0,\pm \theta_{t_j}\bigr\}$, $c_{t_j}\neq 0$},\]
and so does $\hat{h}(x)$. From here and~\eqref{hhat}, there is $\mathcal{I}_1\subset \mathcal{J}_1\subset \mathcal{I}$ and $\epsilon_p\in \{\pm 1\}$ for $p\in \mathcal{J}_1\setminus\mathcal{I}_1$ such that
\begin{equation}\label{00e0q-9}
h(x)=\prod_{p\in \mathcal{J}_1}(x-p)^{\epsilon_{p}\theta_{p}}
\frac{\prod_{j=1}^{m_1}\bigl(x-a_j\bigr)}{\prod_{j=1}^{m_2}\bigl(x-b_j\bigr)},
\end{equation}
where $a_j$'s and $b_k$'s satisfy
\begin{equation*}
 a_j, b_k\in\mathbb{C}\setminus\mathcal{I},\qquad
 \{a_1,\dots, a_{m_1}\}\cap\{b_1,\dots, b_{m_2}\}=\varnothing.
\end{equation*}
In other words,
\[\mathcal{J}_1\setminus\mathcal{I}_1=\{t_{n+j}\mid h(t_{n+j})\in\{0,\infty\},\, 1\leq j\leq m\}.\]
Then by $h(x)=y_{+}(x)/y_{-}(x)$, we see that each $a_j$ (resp.\ $b_j$)
is a zero of $y_+(x)$ (resp.\ $y_-(x)$) and must be simple, namely each
$a_j$ is a simple zero of $h(x)$ and each $b_j$ is a simple pole of
$h(x)$, so $a_j\neq a_k$ and $b_j\neq b_k$ for any $j\neq
k$. Therefore,
\begin{equation*} 
 \parbox{\dimexpr\linewidth-5em}{\it
 Any two elements of the collection $[p\in\mathcal{I};\, a_j, 1\leq j\leq m_1;\, b_k, 1\leq k\leq m_2]$ are distinct. 
}
\end{equation*}
Clearly,~\eqref{h-infty} implies
\begin{equation}\label{00e0q-10-copy}\sum_{p\in \mathcal{J}_1}\epsilon_{p}\theta_{p}+m_1-m_2=-\theta_\infty.\end{equation}

By (\ref{00e0q-9}), we have
{\allowdisplaybreaks
\begin{align}\label{00e0q-11-copy}
h'(x)&=h(x)\Bigg[\sum_{p\in \mathcal{J}_1}\frac{\epsilon_{p}\theta_{p}}{x-p}
+\sum_{j=1}^{m_1}\frac{1}{x-a_j}-\sum_{j=1}^{m_2}\frac{1}{x-b_j}\Bigg]\nonumber\\
&=\frac{h(x)G(x)}{\prod_{p\in\mathcal{J}_1}(x-p)\prod_{j=1}^{m_1}\bigl(x-a_j\bigr)\prod_{j=1}^{m_2}\bigl(x-b_j\bigr)}\nonumber\\
&=\prod_{p\in \mathcal{J}_1}(x-p)^{\epsilon_{p}\theta_{p}-1}\frac{G(x)}{\prod_{j=1}^{m_2}\bigl(x-b_j\bigr)^2},
\end{align}
}%
where $G(x)$ is a polynomial defined by
{\allowdisplaybreaks
\begin{align}
G(x):={}&
\prod_{j=1}^{m_1}\bigl(x-a_j\bigr)\prod_{j=1}^{m_2}\bigl(x-b_j\bigr)\sum_{p\in \mathcal{J}_1}\epsilon_{p}\theta_{p}\prod_{q\in\mathcal{J}_1\setminus\{p\}}(x-q)\nonumber\\
&{}+\prod_{p\in \mathcal{J}_1}(x-p)\prod_{j=1}^{m_2}\bigl(x-b_j\bigr)\sum_{k=1}^{m_1}\prod_{j=1,\neq k}^{m_1}\bigl(x-a_j\bigr)\nonumber\\
&{}-\prod_{p\in \mathcal{J}_1}(x-p)\prod_{j=1}^{m_1}\bigl(x-a_j\bigr)\sum_{k=1}^{m_2}\prod_{j=1,\neq k}^{m_2}\bigl(x-b_j\bigr)\nonumber\\
={}&\biggl(\sum_{p\in \mathcal{J}_1}\epsilon_{p}\theta_{p}+m_1-m_2\biggr)x^{m_1+m_2+|\mathcal{J}_1|-1}+\text{l.o.t.}
= -\theta_{\infty}x^{L}+\text{l.o.t.},\label{00e0q-11}
\end{align}
}%
where $|\mathcal{J}_1|:=\#\mathcal{J}_1\geq n+2$ and
\begin{equation*}
L:=m_1+m_2+|\mathcal{J}_1|-1.
\end{equation*}
Note that the last identity is due to~\eqref{00e0q-10-copy}.

\begin{Lemma}
Suppose $h(x)$ given by~\eqref{00e0q-9} is a developing map of a co-axial solution $u(x)$ of~\eqref{00eq: DE on C}, and let $G(x)$ be defined by~\eqref{00e0q-11}. Then
\begin{gather}
\label{00e0q-13-copy}L=m_1+m_2+|\mathcal{J}_1|-1=\sum_{p\in \mathcal{I}\setminus\mathcal{J}_1}{\alpha}_{p},\\
\label{00e0q-12-copy}G(x)=-\theta_\infty \prod_{p\in\mathcal{I}\setminus \mathcal{J}_1}(x-p)^{{\alpha}_{p}}.
\end{gather}
\end{Lemma}

\begin{proof} Since $h(x)$ is a developing map of $u(x)$, it follows from~\eqref{e22q-1} that
\begin{equation}
\label{00e0q-8}-2Q(x)=\biggl(\frac{h''(x)}{h'(x)}\biggr)'-\frac12 \biggl(\frac{h''(x)}{h'(x)}\biggr)^2.
\end{equation}

Let $G(p)=0$, first we claim that $p\in\mathcal{I}\setminus \mathcal{J}_1$.
Suppose not. Then
\[
 p\notin\{0,1,t_1,\dots,t_{n+m}, a_1,\dots, a_{m_1}, b_1,\dots,
 b_m\}.
\]
Denote by $\gamma\geq 1$ to be the zero order of $G(x)$ at $p$, then $h'(x)=c(x-p)^\gamma(1+O(x-p))$ with $c\neq 0$ and so a direct computation via (\ref{00e0q-8}) implies
\begin{equation}\label{fcfa}-2Q(x)=\frac{-\gamma (\gamma+2)}{2(x-p)^2}+O\bigl((x-p)^{-1}\bigr),\end{equation}
a contradiction with the fact that $Q(x)$ is holomorphic in $\C\setminus\{0,1,t_1,\dots, t_{n+m}\}$. This proves $p\in\mathcal{I}\setminus \mathcal{J}_1$.

Now for each $p\in\mathcal{I}\setminus \mathcal{J}_1$,
we claim that $p$ is a zero of $G(x)$ with order $\alpha_{p}$.
Suppose $h'(x)=c(x-p)^\gamma(1+O(x-p))$ with $c\neq 0$ and
$\gamma\in\mathbb{Z}_{\geq 0}$. Again a direct computation via
(\ref{00e0q-8}) implies~\eqref{fcfa},
so $\gamma(\gamma+2)/4=\beta_{p}=\alpha_{p}(\alpha_{p}+2)/4$, namely
$\gamma=\alpha_{p}$. 
This proves~\eqref{00e0q-13-copy} and~\eqref{00e0q-12-copy}.
\end{proof}

\begin{Corollary}
Suppose that $h(x)$ given by~\eqref{00e0q-9} is a developing map of a
co-axial solution $u(x)$ of~\eqref{00eq: DE on C}. Then $m_1,
m_2\in\mathbb{Z}_{\geq 0}$ are determined by
\begin{gather}\label{m1}
2m_1=\sum_{p\in\mathcal{I}\setminus \mathcal{J}_1}\theta_p-\sum_{p\in\mathcal{J}}\epsilon_p\theta_p-(n+m+1),
\\ \label{m2}
2m_2=\sum_{p\in\mathcal{I}\setminus \mathcal{J}_1}\theta_p+\sum_{p\in\mathcal{J}}\epsilon_p\theta_p-(n+m+1),
\end{gather}
where $\mathcal{J}:=\mathcal{J}_1\cup\{\infty\}$ and $\epsilon_\infty=1$. In particular, $\mathcal{I}\setminus\mathcal{J}_1\neq \varnothing$.
\end{Corollary}

\begin{proof} Note that~\eqref{00e0q-10-copy} and~\eqref{00e0q-13-copy} can be written as
\begin{gather*}
m_1-m_2=-\sum_{p\in\mathcal{J}_1}\epsilon_p\theta_p-\theta_\infty=-\sum_{p\in\mathcal{J}}\epsilon_p\theta_p,\\
m_1+m_2 =\sum_{p\in \mathcal{I}\setminus\mathcal{J}_1}{\alpha}_{p}-|\mathcal{J}_1|+1=\sum_{p\in \mathcal{I}\setminus\mathcal{J}_1}\theta_{t_p}-|\mathcal{I}\setminus\mathcal{J}_1|-|\mathcal{J}_1|+1\\
\hphantom{m_1+m_2}{}
=\sum_{p\in \mathcal{I}\setminus\mathcal{J}_1}\theta_{t_p}-(n+m+1),
\end{gather*}
so~\eqref{m1} and~\eqref{m2} hold, which
also imply $\mathcal{I}\setminus \mathcal{J}_1\neq\varnothing$.
\end{proof}

We now do some preparatory work for the proofs of Theorems~\ref{00thm-2-2} and~\ref{00thm-2-1}.
Denote $\boldsymbol t=(t_1,\dots,t_{n+m})$,
$\boldsymbol
a=(a_1,\dots,a_{m_1})$ and $\boldsymbol b=(b_1,\dots,b_{m_2})$.
Recall that $\mathcal{I}_1\subset\mathcal{J}_1\subsetneqq \mathcal{I}$, we denote
$\boldsymbol t_1=(t_j)_{t_j\in \mathcal{J}_1}$ and $\boldsymbol
t_2=(t_j)_{t_j\in \mathcal{I}\setminus\mathcal{J}_1}$ for
convenience. For example, if
\[
 \mathcal{J}_1=\{0,1,t_1,\dots, t_n, t_{n+1}\dots, t_{n+i}\}
 \qquad\text{for some $0\leq i\leq m-1$},
\]
then $\boldsymbol t_1=(t_1,\dots,t_{n+i})$ and $\boldsymbol
t_2=(t_{n+i+1},\dots, t_{n+m})$. Recall
\[
 \Q_{\theta}:=\Q(\theta_0,\theta_1,\theta_\infty,\theta_1',
 \dots,\theta_n')\subsetneqq\R.
\]
 Then
it follows from~\eqref{00e0q-11} that
\begin{equation}\label{g111111}
G(x)=-\theta_{\infty}\Bigg(x^{L}+\sum_{j=1}^{L}R_j(\boldsymbol a,\boldsymbol b,\boldsymbol t_1)x^{L-j}\Bigg),
\end{equation}
where 
$R_j(\boldsymbol a,\boldsymbol b,\boldsymbol t_1)\in\mathbb{Q}_{\theta}[\boldsymbol a,\boldsymbol b,\boldsymbol t_1] $ is of total degree $j$ for any $j\leq L-1$, while $R_L(\boldsymbol a,\boldsymbol b, \boldsymbol t_1)\in\mathbb{Q}_{\theta}[\boldsymbol a,\boldsymbol b,\boldsymbol t_1] $ is of total degree $L-1$ due to the non-homogenous term $x-1$ in the expression of~$G(x)$.

On the other hand,
\begin{equation}\label{00e0q-12-copy4} \prod_{p\in \mathcal{I}\setminus\mathcal{J}_1}(x-p)^{{\alpha}_{p}}=\prod_{t_j\in \mathcal{I}\setminus\mathcal{J}_1}\bigl(x-t_j\bigr)^{{\alpha}_{t_j}}=x^{L}
+\sum_{j=1}^{L}
C_j(\boldsymbol t_2)x^{L-j},\end{equation}
where we use~\eqref{00e0q-13-copy} and
$C_j(\boldsymbol t_2)\in \Q[\boldsymbol t_2]$ are symmetric polynomials of $\boldsymbol t_2$ of degree $j$.
Comparing~(\ref{g111111}) and~\eqref{00e0q-12-copy4}, we conclude that~\eqref{00e0q-12-copy} holds if and only if $(\boldsymbol a, \boldsymbol b, \boldsymbol t)$ is a common zero of the polynomials
\begin{equation}\label{e0q1-023}B_j(\boldsymbol a,\boldsymbol b,\boldsymbol t):=R_j(\boldsymbol a,\boldsymbol b,\boldsymbol t_1)-C_j(\boldsymbol t_2),\qquad 1\leq j\leq L.\end{equation}
Remark that $B_j(\boldsymbol a,\boldsymbol b,\boldsymbol t)$ is of total degree $j$.

\begin{Definition}
We say that a common zero $(\boldsymbol a,\boldsymbol b,\boldsymbol
t)\in\mathbb{C}^{m_1+m_2+n+m}$ of the polynomials $B_j$'s in
\eqref{e0q1-023} is \emph{admissible} if any two elements of
$(\boldsymbol a,\boldsymbol b,t_1,\dots, t_{n+m})$ do not equal and
none of them equals to $0$, $1$.
\end{Definition}

Then the above argument shows that once~\eqref{00eq: ODE on C} comes
from a co-axial solution $u(x)$ of~\eqref{00eq: DE on C}, then there
exist $\mathcal{I}_1\subset\mathcal{J}_1\subsetneqq \mathcal{I}$ and
$\epsilon_p\in \{\pm 1\}$ for $p\in\mathcal{J}_1$ such that $m_1$, $m_2$
defined by~\eqref{m1} and~\eqref{m2} are non-negative integers,
\eqref{00e0q-13-copy} holds and the corresponding polynomials $B_j$'s
in~\eqref{e0q1-023} has an admissible zero $(\boldsymbol a,\boldsymbol
b,\boldsymbol t)$. The following result shows that the converse
statement also holds.

\begin{Lemma}\label{lem3-5}
Suppose there exist
$\mathcal{I}_1\subset\mathcal{J}_1\subsetneqq \mathcal{I}$ and
$\epsilon_p\in \{\pm 1\}$ for $p\in\mathcal{J}_1$ such that $m_1$, $m_2$
defined by~\eqref{m1} and~\eqref{m2} are non-negative integers,
i.e.,~\eqref{00e0q-13-copy} holds and the corresponding polynomials
$B_j$'s in~\eqref{e0q1-023} are well defined. If $B_j$'s have an
admissible zero $(\boldsymbol a,\boldsymbol b,\boldsymbol t)$, then
\eqref{00eq: DE on C} has a co-axial solution $u(x)$ with its
developing map given by~\eqref{00e0q-9}.
\end{Lemma}

\begin{proof} Under the above assumptions, we consider the function $h(x)$ given by~\eqref{00e0q-9}. Then $h'(x)$ is given by~\eqref{00e0q-11-copy} with $G(x)$ given by~\eqref{00e0q-11}.
Since $(\boldsymbol a,\boldsymbol b,\boldsymbol t)$ is an admissible
zero of polynomials~$B_j$'s, we see that $G(x)$ satisfies
\eqref{00e0q-12-copy}. Inserting~\eqref{00e0q-12-copy} into
\eqref{00e0q-11-copy}, we obtain
\begin{equation}\label{200e0q-12-copy}h'(x)=-\theta_\infty \frac{ \prod_{p\in \mathcal{J}_1}(x-p)^{\epsilon_{p}\theta_{p}-1}\prod_{p\in \mathcal{I}\setminus\mathcal{J}_1}(x-p)^{{\alpha}_{p}}}{\prod_{j=1}^{m_2}\bigl(x-b_j\bigr)^2}.\end{equation}
Furthermore,~\eqref{m1} and~\eqref{m2} imply~\eqref{00e0q-10-copy}, so
\begin{equation}\label{h}
 h(x)=x^{-\theta_\infty}\bigl(1+O\bigl(x^{-1}\bigr)\bigr)\qquad \text{as }x\to\infty.
\end{equation}
It follows that
\begin{equation}\label{h'}h'(x)=-\theta_\infty x^{-\theta_\infty-1}\bigl(1+O\bigl(x^{-1}\bigr)\bigr)\qquad \text{as }x\to\infty.\end{equation}

Define $Q(x)$ via this $h(x)$ by~\eqref{00e0q-8}, i.e., $Q(x)=-\frac12\{h(x),x\}$.
Then by~\eqref{200e0q-12-copy}--\eqref{h'}, a direct computation shows
that
\begin{gather}
\label{eeeqq0}Q(x)=\frac{\theta_p^2-1}{4(x-p)^2}+\frac{d_p}{x-p}+O(1)\qquad \text{near }p\in \mathcal{I}=\{0,1,t_1,\dots,t_{m+n}\},\\
\label{eeeqq00}Q(x)=\frac{\theta_\infty^2-1}{4x^2}+O\bigl(x^{-3}\bigr)\qquad \text{near }x=\infty,\\
\label{eeeqq}Q(x)=\frac{d_{b_j}}{x-b_j}+O(1)\qquad \text{near }x=b_j,
\end{gather}
for some $d_p\in\mathbb{C}$, and $Q(x)$ is holomorphic at elsewhere. We need to prove $d_{b_j}=0$.

Consider the corresponding ODE~\eqref{00eq: ODE on C} with this $Q(x)$. Again by $Q(x)=-\frac12\{h(x),x\}$, a~classical result says that there is a basis $(y_1(x),y_2(x))$ of solutions of (\ref{00eq: ODE on C}) such that $h(x)=y_1(x)/y_2(x)$. From here and the expression (\ref{00e0q-9}) of $h(x)$, we conclude that any solution of (\ref{00eq: ODE on C}) has no logarithmic singularities at any singularities.
On the other hand, by~\eqref{eeeqq} we know that the local exponents of (\ref{00eq: ODE on C}) at $b_j$ are $0$, $1$, so there is a local solution of the following form
\[\tilde{y}(x)=1+c_1\bigl(x-b_j\bigr)+c_2\bigl(x-b_j\bigr)^2+\cdots \qquad\text{near }b_j.\]
Inserting this and~\eqref{eeeqq} into
$\tilde{y}''(x)=Q(x)\tilde{y}(x)$ we immediately obtain $d_{b_j}=0$,
i.e., $Q(x)$ is holomorphic at any $b_j$. Thus it follows from
\eqref{eeeqq0} and \eqref{eeeqq00} that the Riemann scheme of
\eqref{00eq: ODE on C} is~\eqref{00e0q-10}, and~\eqref{00eq: ODE on C}
is apparent at $t_j$ for any $n+1\leq j\leq n+m$.

Let $(y_+(x), y_-(x))$ be the basis of local solutions of~\eqref{00eq: ODE on C} given by~\eqref{00e0q-y+}. Then $(y_1(x),$ $y_2(x))=(y_+(x), y_-(x))\SM{a}{b}{c}{d}$ for some $ad-bc\neq 0$, so $h(x)=\frac{ay_+(x)+cy_-(x)}{by_+(x)+dy_-(x)}$. Since under the basis $(y_{+}(x), y_{-}(x))$, we have $M_{\infty}=\SM{{\rm e}^{-\pi {\rm i} \alpha_\infty}}{0}{0}{{\rm e}^{\pi {\rm i} \alpha_\infty}}$, so we see from~\eqref{h} that
\[{\rm e}^{-2\pi {\rm i}
 \alpha_{\infty}}\frac{ay_+(x)+cy_-(x)}{by_+(x)+dy_-(x)}={\rm e}^{-2\pi {\rm i}
 \theta_{\infty}}h(x)=\frac{a{\rm e}^{-2\pi {\rm i}
 \alpha_{\infty}}y_+(x)+cy_-(x)}{b{\rm e}^{-2\pi {\rm i}
 \alpha_{\infty}}y_+(x)+dy_-(x)}.\] From here and ${\rm e}^{-2\pi {\rm i}
 \alpha_{\infty}}\neq 1$, we easily obtain $b=c=0$ and so
$h(x)=ay_+(x)/dy_-(x)=y_+(x)/y_-(x)$, where $a/d=1$ follows from
\eqref{h} and~\eqref{00e0q-y+}. Then by Lemma~\ref{lemm3-1}, we
conclude that~\eqref{00eq: DE on C} has a co-axial solution $u(x)$
with $h(x)$ being its developing map.
\end{proof}

Therefore, the problem turns to study the admissible zeros of the
polynomials $B_j$'s. The number of the polynomials is
$L=m_1+m_2+|\mathcal{J}_1|-1$ with $|\mathcal{J}_1|\leq
|\mathcal{I}|-1=n+m+1$, and the number of variables is
$m_1+m_2+n+m\geq L$. Thus there are two different cases: (i)
$m_1+m_2+n+m=L$ and (ii) $m_1+m_2+n+m>L$. For the case (i), we may
expect that the number of common zeros of the polynomials $B_j$'s is
finite, but we can not expect the validity of this assertion for the
case (ii).

Therefore, let us study the case $m_1+m_2+n+m=L$ first. In this case,
$|\mathcal{J}_1|= |\mathcal{I}|-1=n+m+1$, so
$\mathcal{I}\setminus\mathcal{J}_1$ consists of a single point
belonging to $\{t_{n+j}\mid 1\leq j\leq m\}$. Without loss of
generality, we may assume
\begin{equation}\label{fqfa}
 \mathcal{J}_1=\{0,1,t_1,\dots, t_{n+m-1}\},\qquad
 \mathcal{I}\setminus\mathcal{J}_1=\{t_{n+m}\}.
\end{equation}
Then $\boldsymbol t_1=(t_1,\dots, t_{n+m-1})$,~\eqref{00e0q-12-copy4} becomes
\begin{equation*}
\prod_{p\in \mathcal{I}\setminus\mathcal{J}_1} (x-p)^{\alpha_{p}}=(x-t_{n+m})^{{\alpha}_{t_{n+m}}}=x^{L}
+\sum_{j=1}^{L}(-1)^j
\begin{pmatrix}L\\ j\end{pmatrix}t_{n+m}^jx^{L-j},
\end{equation*}
with
\begin{equation}\label{00e0q-13-copy2}L=m_1+m_2+n+m={\alpha}_{t_{n+m}}\end{equation} and so the polynomials $B_j$'s become
\begin{equation}
\label{e0q1-231}
B_j(\boldsymbol a,\boldsymbol b,\boldsymbol t):=R_j(\boldsymbol a,\boldsymbol b,\boldsymbol t_1)-\sum_{j=1}^{L}(-1)^j
\begin{pmatrix}L\\ j\end{pmatrix}t_{n+m}^j,\qquad 1\leq j\leq L.
\end{equation}
Define
\[\mathcal{A}=\mathcal{A}_{(\epsilon_p)}:=\left\{\boldsymbol {t}\in\C^{n+m}\,\Big|\,\begin{array}{@{}lll@{}}\text{there is $\boldsymbol a$, $\boldsymbol b$ such that $(\boldsymbol a,\boldsymbol b,\boldsymbol t)$ is}\\ \text{an admissible zero of the polynomials $B_j$'s in~\eqref{e0q1-231}}\end{array}\!\!\right\}.
\]
Then Lemma~\ref{lem3-5} shows $\mathcal{A}\subset A$, where $A$ is
defined by~\eqref{eq: set A}.
The first main result is as follows.

\begin{Theorem}\label{thm-J1}
Suppose that for $\mathcal{J}_1$ given by~\eqref{fqfa}, there exist
$\epsilon_p\in \{\pm 1\}$ for $p\in\mathcal{J}_1$ such that $m_1$, $m_2$
defined by~\eqref{m1} and \eqref{m2} are non-negative integers,
i.e.,~\eqref{00e0q-13-copy2} holds and the corresponding polynomials
$B_j$'s in~\eqref{e0q1-231} are well defined.
Suppose that $\mathcal{A}\neq \varnothing$. Then $\mathcal{A}$ is a~finite set with $\#\mathcal{A}\leq
\frac{(\theta_{t_{n+m}}-1)!}{m_1!m_2!}$.

Furthermore, for each $\boldsymbol t=(t_1,\dots,t_{n+m})\in
\mathcal{A}$, all $t_j$'s are algebraic over $\Q_{\theta}$, and the
field
\[
 \Q_\theta(\mathcal{A}):=\Q_{\theta}(\{t_j \mid (t_1,\dots,t_{n+m})\in
 \mathcal{A}\})
\]
is a Galois extension of $\Q_{\theta}$ and
\[[\Q_\theta(\mathcal{A}) :\Q_\theta]\leq M<\infty,\]
where $M$ is a constant depending on the integer angles
$\theta_{t_{j}}$'s.
\end{Theorem}

\begin{proof}
Given any $\boldsymbol t_0=(t_{0,1},\dots,t_{0,n+m})\in \mathcal{A}$. Then the polynomials $B_j$'s in~\eqref{e0q1-231} have an admissible zero $(\boldsymbol a,\boldsymbol b,\boldsymbol t)$ such that $\boldsymbol t=\boldsymbol t_0$. We claim that
\begin{equation} \label{poly=m1}
 \parbox{\dimexpr\linewidth-5em}{\it
 The polynomial system $B_j(\boldsymbol a,\boldsymbol b,\boldsymbol t)=0$ with $1\leq j\leq L$\\ has at most $L!=(\theta_{t_{n+m}}-1)!$ solutions.
 }
\end{equation}
Once this claim is proved, we can conclude that $\mathcal{A}$ is a
finite set with $\# \mathcal{A}\leq
\frac{(\theta_{t_{n+m}}-1)!}{m_1!m_2!}$ due to the fact that $B_j$ is
invariant under any permutation of $a_1,\dots, a_{m_1}$ and any
permutation of $b_1,\dots, b_{m_2}$ for all $j$.

To prove the claim~\eqref{poly=m1}, we consider the homogenization of this polynomial system
\begin{equation}\label{00e0q-22-copy-1}
\begin{split}
\widetilde{B}_j(\boldsymbol a, \boldsymbol b,\boldsymbol t,\varepsilon):=\varepsilon^jB_j({\boldsymbol a}/\varepsilon, {\boldsymbol b}/\varepsilon, \boldsymbol t/\varepsilon)=0,\qquad j=1,\dots, L.
\end{split}
\end{equation}
We show that the solution of~\eqref{00e0q-22-copy-1} with $\varepsilon=0$ must be $\boldsymbol 0$,
i.e., $(\boldsymbol a, \boldsymbol b,\boldsymbol t)=\boldsymbol 0$.

First of all, since $0,1\in\mathcal{J}_1$, we see from~\eqref{00e0q-11} and~\eqref{e0q1-231} that
\begin{equation*}
B_{L}(\boldsymbol a, \boldsymbol b,\boldsymbol t)=\frac{\epsilon_0\theta_0}{\theta_\infty} \prod_{j=1}^{n+m-1}t_j\prod_{j=1}^{m_1}a_j\prod_{j=1}^{m_2}b_j-(-1)^{L}
t_{n+m}^{L},
\end{equation*}
and so
\[\widetilde{B}_{L}(\boldsymbol a, \boldsymbol b,\boldsymbol t,\varepsilon)=\frac{\epsilon_0\theta_0}{\theta_\infty} \varepsilon \prod_{j=1}^{n+m-1}t_j\prod_{j=1}^{m_1}a_j\prod_{j=1}^{m_2}b_j-(-1)^{L}
t_{n+m}^{L}.\]
Thus $\widetilde{B}_{L}(\boldsymbol a, \boldsymbol b,\boldsymbol t,0)=0$ implies $t_{n+m}=0$.

To prove that $(\boldsymbol a, \boldsymbol b, t_1,\dots,t_{n+m-1})=\boldsymbol 0$, we define
\begin{equation*}
\widetilde{h}(x)=x^{\epsilon_0\theta_0+\epsilon_1\theta_1}\prod_{j=1}^{n+m-1}\bigl(x-t_j\bigr)^{\epsilon_{t_j}\theta_{t_j}}
\frac{\prod_{j=1}^{m_1}\bigl(x-a_j\bigr)}{\prod_{j=1}^{m_2}\bigl(x-b_j\bigr)},
\end{equation*}
i.e., replace the non-homogenous factor $x-1$ of $h(x)$ in~\eqref{00e0q-9} with $x$.
Then
\begin{align*}
\widetilde{h}'(x)=x^{\epsilon_0\theta_0+\epsilon_1\theta_1-2}\prod_{j=1}^{n+m-1}\bigl(x-t_j\bigr)^{\epsilon_{t_j}\theta_{t_j}-1}\frac{\widetilde{G}(x)}{\prod_{j=1}^{m_2}\bigl(x-b_j\bigr)^2},
\end{align*}
where the expression of $\widetilde{G}(x)$ is obtained from that of $G(x)$ in~\eqref{00e0q-11} by replacing the term $x-1$ with $x$. In other words,
\[\widetilde{G}(x)=-\theta_{\infty}\Bigg(x^{L}+\sum_{j=1}^{L}\widetilde{R}_j(\boldsymbol a,\boldsymbol b,\boldsymbol t_1)x^{L-j}\Bigg),\]
where $\widetilde{R}_j(\boldsymbol a,\boldsymbol b,\boldsymbol t_1)$
is the homogeneous part of degree $j$ of $R_j(\boldsymbol
a,\boldsymbol b,\boldsymbol t)$, and so
\[
\widetilde{B}_j(\boldsymbol a,\boldsymbol b,\boldsymbol t_1,t_{n+m},0)=\widetilde{R}_j(\boldsymbol a,\boldsymbol b,\boldsymbol t_1)-(-1)^j
\begin{pmatrix}L\\ j\end{pmatrix}t_{n+m}^j,\qquad 1\leq j\leq L.
\]
Thus $\widetilde{B}_j(\boldsymbol a,\boldsymbol b,\boldsymbol t_1,0,0)=0$ for any $j$ yields $\widetilde{R}_j(\boldsymbol a,\boldsymbol b,\boldsymbol t_1)=0$ for any $j$ and so
\[
\widetilde{G}(x)=-\theta_{\infty}x^{L}=-\theta_\infty x^{m_1+m_2+n+m}.
\]
Suppose the number of those $a_j$'s being $0$ is $n_1$, the number of those $b_j$'s being $0$ is $n_2$ and denote
\[
I_2=\{j\in \{1,\dots, n+m-1\} \mid t_j=0\}.
\]
Then
\[
\widetilde{h}(x)=x^{\epsilon_0\theta_0+\epsilon_1\theta_1+\sum_{j\in I_2}\epsilon_{t_j}\theta_{t_j}+n_1-n_2}(c+o(x))
\]
and so
\[
\widetilde{h}'(x)=x^{\epsilon_0\theta_0+\epsilon_1\theta_1+\sum_{j\in I_2}\epsilon_{t_j}\theta_{t_j}+n_1-n_2-1}(c'+o(x))
\]
for some $c, c'\neq 0$. Therefore,
\begin{gather*}
\epsilon_0\theta_0+\epsilon_1\theta_1+\sum_{j\in I_2}\epsilon_{t_j}\theta_{t_j}+n_1-n_2-1\\
\qquad\quad=\epsilon_0\theta_0+\epsilon_1\theta_1-2+\sum_{j\in I_2}(\epsilon_{t_j}\theta_{t_j}-1)+m_1+m_2+n+m-2n_2,
\end{gather*}
namely $n_1+n_2+\# I_2=m_1+m_2+n+m-1$. Since $n_j\leq m_j$ and $\#I_2\leq n+m-1$, we obtain $n_1=m_1$, $n_2=m_2$ and $\# I_2=n+m-1$, i.e., $(\boldsymbol a, \boldsymbol b, \boldsymbol t_1)=\boldsymbol 0$.

Therefore, the homogenized system (\ref{00e0q-22-copy-1}) has no
solutions at infinity. Then by the Bezout theorem, the polynomial
system $B_j(\boldsymbol a,\boldsymbol b,\boldsymbol t)=0$ with $1\leq
j\leq L$ has exactly $L!$
solutions by counting multiplicity. This proves the claim (\ref{poly=m1}). Furthermore, it follows from Lemma~\ref{00lemma-2-10} below that for any solution $(\boldsymbol a, \boldsymbol b, \boldsymbol t)$, every element of $(\boldsymbol a, \boldsymbol b, \boldsymbol t)$ is algebraic over $\Q_{\theta}$, i.e., every element belongs to $\overline{\Q_{\theta}}$, where
$\overline{\Q_{\theta}}$ denotes the algebraic closure of $\Q_{\theta}$.

Let $(\boldsymbol a, \boldsymbol b, \boldsymbol t)$ be an admissible
zero of the polynomials $B_j$'s.
Let $\sigma\colon \overline{\Q_{\theta}}\to \overline{\Q_{\theta}}$ be any automorphism of $\overline{\Q_{\theta}}$ such that $\sigma(x)=x$ for any $x\in \Q_{\theta}$. Then it follows from $B_j(\boldsymbol a,\boldsymbol b,\boldsymbol t)\in\mathbb{Q}_{\theta}[\boldsymbol a,\boldsymbol b,\boldsymbol t] $ that
\[
 (\sigma(a_1),\dots,\sigma(a_{m_1}),\sigma(b_1),\dots,
 \sigma(b_{m_2}),\sigma(t_1),\dots,\sigma(t_{n+m}))
\]
is also an admissible zero of the polynomials $B_j$'s, namely $(
\sigma(t_1),\dots,\sigma(t_{n+m}))\in \mathcal{A}$ as long as $(t_1,\dots, t_{n+m})\in \mathcal{A}$.
This proves that $\Q_{\theta}(\mathcal{A})$ is a Galois extension of $\Q_{\theta}$, and $[\Q_{\theta}(\mathcal{A}):\Q_{\theta}]$ is bounded by a constant depending on the integer angles $\theta_{t_{j}}$'s, because the degree of the minimal polynomial of each $t_j$ is bounded by $(\theta_{t_{n+m}}-1)!$ and $\# \mathcal{A}\leq (\theta_{t_{n+m}}-1)!$. The proof is complete.
\end{proof}

\begin{Lemma}\label{00lemma-2-10} Given complex numbers
 $y_1,\dots,y_k$, let $K$
 be the algebraic closure of the field $\Q(y_1,\dots,y_k)$.
Let $H_j(x_1,\dots,x_\ell)\in K[x_1,\dots,x_\ell]$ for
$j=1,\dots,\ell$. Suppose that the polynomial system
\begin{equation}\label{e0q-36}
 H_j(x_1,\dots,x_\ell)=0,\qquad 1\leq j\leq \ell,
\end{equation}
has only finitely many solutions. Then the coordinates of any solution
$(t_1,\dots,t_\ell)$ of~\eqref{e0q-36} are algebraic over
$\Q(y_1,\dots,y_\ell)$.
\end{Lemma}

\begin{proof} Let $I$ be the ideal of $K[x_1,\dots,x_\ell]$ generated
 by the polynomials $H_j$. For $i=1,\dots,\ell$, consider the
 elimination ideal $I_i=I\cap K[x_i]$. Under the assumption that
~\eqref{e0q-36} has finitely many solutions, $I_i$ is generated by
 some polynomial $f_i(x_i)\in K[x_i]$ whose roots are precisely those
 $t_i$ such that $(t_1,\dots,t_\ell)$ is a solution of
~\eqref{e0q-36} for some $(t_1,\dots,t_{i-1},t_{i+1},\dots,t_\ell)$,
 by the closure theorem in the elimination theory (see \cite[Theorem
 3, Section 3.3]{Cox}). Since all roots of $f_i$ are in $K$, this
 shows that the coordinates of any solution
 $(t_1,\dots,t_\ell)$ of~\eqref{e0q-36} are algebraic over
 $\Q(y_1,\dots,y_\ell)$.
\end{proof}

\subsection[The case m=1]{The case $\boldsymbol{m=1}$}
This section is devoted to the proof of Theorem~\ref{00thm-2-2}.

\begin{proof}[Proof of Theorem~\ref{00thm-2-2}]
Let $m=1$ and~\eqref{00e0qq-111}--\eqref{k3} with
$\theta_{t_j}=\theta_j'$ hold, i.e., the set $A$ in~\eqref{eq: set A}
is non-empty.
Given any $\boldsymbol t_0=(t_{0,1},\dots,t_{0,n+1})\in A$. Then
there exist
$\mathcal{I}_1\subset\mathcal{J}_1\subsetneqq \mathcal{I}$ and $\epsilon_p\in \{\pm 1\}$ for $p\in\mathcal{J}_1$ such that $m_1$, $m_2$ defined by~\eqref{m1} and \eqref{m2} are non-negative integers,~\eqref{00e0q-13-copy2} holds and the corresponding polynomials $B_j$'s in~\eqref{e0q1-023} has an admissible zero $(\boldsymbol a,\boldsymbol b,\boldsymbol t)$ with $\boldsymbol t=\boldsymbol t_0$.
Since $m=1$ implies $\mathcal{I}_1\cup\{t_{n+m}\}=\mathcal{I}$, we always have
\[\mathcal{J}_1=\mathcal{I}_1=\{0,1,t_1,\dots,t_n\}\qquad \text{and}\qquad\mathcal{I}\setminus \mathcal{I}_1=\{t_{n+1}\},\]
i.e., it satisfies~\eqref{fqfa} and so the polynomials $B_j$'s become~\eqref{e0q1-231}. Then by Theorem~\ref{thm-J1}, we have $\boldsymbol t_0\in\mathcal{A}_{(\epsilon_p)}$ and $\#\mathcal{A}_{(\epsilon_p)}\leq \frac{(\theta_{t_{n+1}}-1)!}{m_1!m_2!}=\frac{(\theta_{n+1}'-1)!}{m_1!m_2!}.$

Remark that in general, the set
\[
 \Lambda:=\left\{(\epsilon_p)_{p\in\mathcal{I}_1}\,\Big|\,
 \begin{array}{@{}lll@{}}\text{$\epsilon_p\in\{\pm 1\}$, $m_1$, $m_2$ defined by~\eqref{m1} and \eqref{m2} are}\\
 \text{non-negative integers and so~\eqref{00e0q-13-copy2} holds}
 \end{array}\!\!\right\}
\]
might contain multiple elements. Of course $\#\Lambda\leq
2^{n+2}$. Since the above argument implies
\[A=\bigcup\limits_{(\epsilon_p)\in \Lambda}\mathcal{A}_{(\epsilon_p)},\]
we conclude that $A$ is a finite set with $\# A\leq
2^{n+2}\bigl(\theta_{n+1}'-1\bigr)!$ and $\Q_\theta(A)$ is a finite Galois
 extension of $\Q_\theta$ with $[\Q_\theta(A):\Q_\theta]\leq M'$,
 where $M'$ is a constant depending on $n$ and the integer angle
 $\theta_{n+1}'$. The proof is complete.
\end{proof}





%



\subsection[The general case m geq 2]{The general case $\boldsymbol{m\geq 2}$}

Since $n+2\leq |\mathcal{J}_1|\leq n+m+1$, we have $m_1+m_2+n+m\geq L=m_1+m_2+|\mathcal{J}_1|-1$, and in particular, the case $m_1+m_2+n+m>L$ appears generally, so we can not expect that the corresponding polynomial system
\begin{equation}
\label{e0q1-23}
B_j(\boldsymbol a,\boldsymbol b,\boldsymbol t)=R_j(
\boldsymbol a,\boldsymbol b,\boldsymbol t_1)-C_j(\boldsymbol t_2)=0,
\qquad 1\leq j\leq L
\end{equation}
has only finitely many solutions. This is the difference from the $m=1$ case.
Our key observation is following.

\begin{Lemma}\label{00lemm3-4}
The solution set $W$ of the polynomial system~\eqref{e0q1-23}
 in $\C^{m_1+m_2+n+m}$ is an affine algebraic set of dimension
 $n+m+1-|\mathcal{J}_1|\leq m-1$.

 Consequently, the dimension of the set of admissible zeros of the
 polynomials $B_j$'s is $\leq m-1$.
\end{Lemma}

\begin{proof} Note that
\[W=\bigl\{(\boldsymbol a, \boldsymbol b, \boldsymbol t_1, \boldsymbol t_2)\in \C^{m_1+m_2+n+m}\mid B_j(\boldsymbol a, \boldsymbol b, \boldsymbol t_1, \boldsymbol t_2)=0,\ \forall j\bigr\},\]
and
\[\widetilde{W}:=W\setminus\bigl\{(\boldsymbol a, \boldsymbol b, \boldsymbol t_1, \boldsymbol t_2)\in \C^{m_1+m_2+n+m}\mid t_j=0\ \text{for some $t_j$ in $\boldsymbol t_2$}\bigr\}\]
is an open subset of $W$.
Fix any $(n+m+2-|\mathcal{J}_1|)$-tupe
\[
\boldsymbol t_{2,0}=(t_{j,0})\in \bigl\{\boldsymbol t_2\in \{1\}\times\C^{n+m+1-|\mathcal{J}_1|}\mid t_{j}\neq 0\ \text{for any $t_j$ in $\boldsymbol t_2$}\bigr\},
\]
(i.e., the first component of $\boldsymbol t_{2,0}$ is $1$.) we claim that
\begin{equation} \label{00eq: claim}
 \parbox{\dimexpr\linewidth-5em}{\it
 there are $L!$ points $($by counting multiplicities$)$ in $W$ satisfying $\boldsymbol t_2=t\boldsymbol t_{2,0}$\\ for some~${t\in\mathbb{C}}$.
 }
\end{equation}
Once this claim is proved, then the dimension of $W$ is $n+m+1-|\mathcal{J}_1|\leq m-1$.

To prove this claim, we insert $\boldsymbol t_2=t\boldsymbol t_{2,0}$ into (\ref{00e0q-12-copy4}), we obtain
\begin{equation*}
\prod_{t_j\in \mathcal{I}\setminus\mathcal{J}_1}\bigl(x-t_{j,0}t\bigr)^{{\alpha}_{t_j}}
=x^{L}
+\sum_{j=1}^{L}
C_j(\boldsymbol t_{2,0})t^jx^{L-j}.\end{equation*}
Then the polynomial system~\eqref{e0q1-23} becomes
\begin{equation}\label{00e0q-23-copy}
\widehat{B}_j(\boldsymbol a, \boldsymbol b,\boldsymbol t_1, t):=R_{j}(\boldsymbol a, \boldsymbol b,\boldsymbol t_1)-C_j(\boldsymbol t_{2,0})t^j=0,\qquad j=1,\dots, L.
\end{equation}
The advantage of this new polynomial system~\eqref{00e0q-23-copy} is that
the numbers of unknowns and equations are the same. Note that $C_{L}(\boldsymbol t_{2,0})=(-1)^{L}\prod t_{j,0}^{{\alpha}_{t_j}}\neq 0$.
Then the same proof as~\eqref{00e0q-22-copy-1} shows that the homogenized system of~\eqref{00e0q-23-copy} has no solutions at infinity, so it follows from the Bezout theorem that the polynomial system (\ref{00e0q-23-copy}) has exactly
$L!$
solutions by counting multiplicity. This proves the claim (\ref{00eq: claim}).
\end{proof}

\begin{proof}[Proof of Theorem~\ref{00thm-2-1}] Due to the finite choices of $\mathcal{J}_1$ and $(\epsilon_p)_{p\in\mathcal{J}_1}$, the assertion that the dimension of $A$ is $\leq m-1$ follows from Lemma~\ref{00lemm3-4}.

Now we further assume (\ref{000eq: theta alpha a}). Then
\[
 \Q_{\theta}=\Q\bigl(\theta_0,\theta_1,\theta_\infty,\theta_1',
 \dots,\theta_n'\bigr)\subsetneqq \overline{\Q}.
\]
Given any $\boldsymbol t_0=\bigl(\tilde{t}_{1},\dots, \tilde{t}_{n+m}\bigr)
\in A$, i.e., (\ref{00eq: DE on C}) has a co-axial solution for some
$\boldsymbol t_0=\bigl(\tilde{t}_{1},\dots, \tilde{t}_{n+m}\bigr)$ satisfying
$\tilde{t}_{j}\neq 0,1$ for any $j$ and $\tilde{t}_{j}\neq
\tilde{t}_{k}$ for any $j\neq k$. Without loss of generality, we may
assume the corresponding
\[
 \mathcal{J}_1=\{0,1,t_1,\dots, t_{n+i}\},\qquad
 |\mathcal{J}_1|=n+i+2,
\]
\[
 \mathcal{I}\setminus \mathcal{J}_1=\{t_{n+i+1},\dots,t_{n+m}\},
\]
for some $0\leq i\leq m-1$. Denote
\[
 \boldsymbol t_{1,0}=\bigl(\tilde{t}_1,\dots,\tilde{t}_{n+i}\bigr),
 \qquad \boldsymbol
t_{2,0}=\Big(1,\tfrac{\tilde{t}_{n+i+2}}{\tilde{t}_{n+i+1}},\dots,
\tfrac{\tilde{t}_{n+m}}{\tilde{t}_{n+i+1}}\Big).
\]
Then there is $(\boldsymbol a_0, \boldsymbol b_0)$ such that
$(\boldsymbol a, \boldsymbol b, \boldsymbol t_1, t)=\bigl(\boldsymbol a_0,
\boldsymbol b_0, \boldsymbol t_{1,0}, \tilde{t}_{n+i+1}\bigr)\in
\C^{m_1+m_2+|\mathcal{J}_1|-1}$ is a solution of
(\ref{00e0q-23-copy}). Since $\Q_{\theta}\subset\overline{\Q}$
implies
\[
 \widehat{B}_j(\boldsymbol a, \boldsymbol b,\boldsymbol t_2, t)
 \in\overline{\Q}\Big[\tfrac{\tilde{t}_{n+i+2}}{\tilde{t}_{n+i+1}},
 \dots, \tfrac{\tilde{t}_{n+m}}{\tilde{t}_{n+i+1}}\Big][\boldsymbol a,
 \boldsymbol b,\boldsymbol t_2,t],
\]
it follows from Lemma~\ref{00lemma-2-10} that $\bigl(\boldsymbol t_{1,0},
\tilde{t}_{n+i+1}\bigr)$ is algebraic over
$\overline{\Q}\Big[\tfrac{\tilde{t}_{n+i+2}}{\tilde{t}_{n+i+1}},\dots,
\tfrac{\tilde{t}_{n+m}}{\tilde{t}_{n+i+1}}\Big]$, namely, the
transcendence degree of $\overline{\Q}\bigl(\tilde{t}_{1},\dots,
\tilde{t}_{n+m}\bigr)$ over $\overline{\Q}$ is $\leq m-1$. The proof is
complete.
\end{proof}

\section{Eremenko's theorem}
\label{sec-Eremenko}

Eremenko's Theorem~\ref{thm-Eremenko} is a deep result, and in this section, we would like to make some discussions about it and prove Theorem~\ref{thm-MP-E}.
By using the notions in Section~\ref{sec-dimension}, we suppose that there exist $\mathcal{I}_1\subset \mathcal{J}_1\subsetneqq \mathcal{I}$ and $\epsilon_p\in\{\pm 1\}$ for $p\in\mathcal{J}:=\mathcal{J}_1\cup\{\infty\}$ such that
\begin{align}\label{erem1}
&k':=\sum_{p\in\mathcal{J}}\epsilon_p \theta_p\in\mathbb{Z}_{\geq 0},\\
\label{erem2}
&k'':=\sum_{p\in\mathcal{I}\setminus\mathcal{J}_1}\theta_p-k'-(n+m+1)\in 2\mathbb{Z}_{\geq 0}.
\end{align}
Set
\[
 \boldsymbol c:=(\underbrace{\theta_0,\theta_1,\theta_\infty,
 \theta_{t_1},\dots,\theta_{p}}
 _{p\text{ runs over }\mathcal{J}},\underbrace{1,\dots,1}_{k'+k''}).
\]
Suppose that
\begin{equation} \label{erem3}
\begin{aligned}
& \parbox{\dimexpr\linewidth-5em}{\it
 Either $\boldsymbol c$ is incommensurable or there is $\eta\neq0$
 such that $\mathbf{c}=\eta\mathbf{b}=\eta(b_1,\dots,b_q)$\\ with
 $b_j\in\mathbb{N}$, $\gcd (b_1,\dots,b_q)=1$ and}
\\
& 2\max_{j\in\mathcal{I}\setminus\mathcal{J}_1}\theta_j\leq \sum_{j=1}^q b_j.
\end{aligned}
\end{equation}
Then the proof of Theorem~\ref{thm-Eremenko} in \cite{Coaxial} actually yields the following result, a more precise form of Theorem~\ref{thm-Eremenko}.

\begin{Theorem}[{\cite{Coaxial}\label{thm-Eremenko2}}]
Equation~\eqref{00eq: DE on C} admits a co-axial solution for some singular set $\{t_1,\dots,\allowbreak t_{n+m}\}$ with the developing map
\[h(x)=\prod_{p\in\mathcal{J}_1}(x-p)^{\epsilon_p\theta_p}
\frac{\prod_{j=1}^{m_1}\bigl(x-a_j\bigr)}{\prod_{j=1}^{m_2}\bigl(x-b_j\bigr)}\]
if and only if~\eqref{erem1}--\eqref{erem3} hold.
\end{Theorem}

It is not difficult to see that the conditions
\eqref{00e0qq-111}--\eqref{k3} in Theorem~\ref{thm-Eremenko} imply
\eqref{erem1}--\eqref{erem3} in Theorem~\ref{thm-Eremenko2}. Theorem~\ref{thm-Eremenko2} shows that under the conditions
\eqref{erem1}--\eqref{erem3}, the polynomials $B_j$'s in~\eqref{e0q1-023} always have an admissible zero. This is really a
remarkable result.

\begin{Example}
Consider the case
$\{\theta_0,\theta_1,\theta_\infty,\theta_t\}
=\bigl\{\frac12,\frac13,\frac16,\theta_t\bigr\}$
with $\theta_t\in\mathbb{N}_{\geq 2}$. If $\theta_t$ is small, then it
is easy to calculate those $t$'s such that~\eqref{00eq: DE on C} has
co-axial solutions. For example, if $\theta_t=2$, then $t=\frac23$; if
$\theta_t=3$, then $t=\frac{2\pm 2 {\rm i}}{3}$. When $\theta_t\geq 4$, the
polynomial system $B_j=0$ becomes very complicated. In Section~\ref{section: proof of Theorem 1.4}, we will give more examples.
\end{Example}

\begin{Example}
Consider the case $\{\theta_0,\theta_1,\theta_\infty,\theta_{t_1},\theta_{t_2}\}=\bigl\{\frac12,\frac16,\frac16,\frac16,\theta_{t_2}\bigr\}$ with $\theta_{t_2}\in\mathbb{N}_{\geq 2}$. If $\theta_{t_2}=3$, then $t_2=\frac{1\pm {\rm i}\sqrt{3}}{2}$.
\end{Example}

Now we turn back to the condition~\eqref{k3} in Theorem~\ref{thm-Eremenko} if the vector $\boldsymbol c$ is commensurable. Under the conditions~\eqref{00e0qq-111}--\eqref{00e0qq-112}, it is not difficult to see that~\eqref{k3} holds automatically provided $k'+k''\geq 1$ and $\theta_{t_j}>1$ for all $j$. A more interesting thing is the assertion of Theorem~\ref{thm-MP-E}, which says that~\eqref{k3} holds provided that Mondello--Panov's condition~\eqref{00e0q-5} holds.

\begin{proof}[Proof of Theorem~\ref{thm-MP-E}]
Here we prove a more general result than Theorem~\ref{thm-MP-E}. To
simplify the notations, we reformulate this problem. Given $n\geq 1$
positive non-integer numbers $\theta_1,\dots,\theta_n$ and $m-n\geq
0$ positive integer numbers $\theta_{n+1},\dots,\theta_{m}\in\N_{\geq
 2}$.
Suppose
\begin{equation}\label{200e0q-5}
d_1\bigl(\Z_o^{m},{\boldsymbol\theta-\boldsymbol 1}\bigr)=1,
\end{equation}
and there are $\epsilon_{j}\in\{\pm 1\}$ for $1\leq j\leq n$ such that
\begin{gather}\label{200e0qq-111}
k':=\sum_{j=1}^n \epsilon_{j}\theta_{j}\in\mathbb{Z}_{\geq 0},
\\ \label{200e0qq-112}
k'':=\sum_{j=n+1}^m \theta_{j}-m-k'+2\in 2\mathbb{Z}_{\geq 0}.
\end{gather}
Let
\[
 \boldsymbol c:=(\theta_1,\dots,\theta_{n},
 \underbrace{1,\dots,1}_{k'+k''}),
\]
and suppose $\boldsymbol c$ is commensurable, i.e., there is $\eta\in\mathbb{R}_{>0}$ such that $\mathbf{c}=\eta\mathbf{b}=\eta(b_1,\dots,b_q)$ with $b_j\in\mathbb{N}$ and $\gcd (b_1,\dots,b_q)=1$.
We want to prove that
\begin{equation}\label{200k3}2\max_{n+1\leq j\leq m}\theta_j\leq \sum_{j=1}^q b_j=\sum_{j=1}^nb_j+\bigl(k'+k''\bigr)\eta^{-1}.\end{equation}
Then Theorem~\ref{thm-MP-E} is equivalent to~\eqref{200k3} with the case $n\geq 3$.

Clearly,~\eqref{200e0q-5} implies $n\geq 2$ and~\eqref{200e0qq-111} and \eqref{200e0qq-112} imply $m-n\geq 1$. By renaming $\theta_j$'s if necessary, we may always assume
\[0<\theta_1\leq \theta_2\leq\cdots\leq \theta_n,\qquad 2\leq\theta_{n+1}\leq \theta_{n+2}\leq\cdots\leq \theta_m.\]
By~\eqref{200e0qq-112}, we have
\[k''+k'=\sum_{j=n+1}^m \theta_{j}-m+2\geq \theta_{m}+2(m-n-1)-m+2,\]
i.e.,
\[\max_{n+1\leq j\leq m}\theta_j=\theta_m\leq k'+k''-m+2n\leq k'+k''+n-1,\]
so to prove~\eqref{200k3}, it suffices to prove
\begin{equation}\label{2200k3}2n-2\leq \sum_{j=1}^n b_j+\bigl(\eta^{-1}-2\bigr)\bigl(k'+k''\bigr).\end{equation}

On the other hand, since
$\mathbf{c}=\eta\mathbf{b}=\eta(b_1,\dots,b_q)$ with
$b_j\in\mathbb{N}$, we have $\eta\neq 1$, which implies
$\eta^{-1}\in\mathbb{N}_{\geq 2}$ if $k'+k''\geq 1$. Thus we always
have
\[
 \bigl(\eta^{-1}-2\bigr)\bigl(k'+k''\bigr)\geq 0\qquad\text{and}\qquad \sum_{j=1}^n
 b_j\geq n.
\]
This implies that (\ref{2200k3}) holds if $n=2$. Thus it remains to consider the case $n\geq 3$. We consider two cases separately.

{\bf Case 1.} $k'+k''\geq 1$. This is a simple case.
Then $N:=\eta^{-1}\in\mathbb{N}_{\geq 2}$ and for $1\leq j\leq n$, $\theta_j=b_j/N\notin\mathbb{Z}$. Consequently,
$\operatorname{dist}(\theta_j,\mathbb{Z}):=\min_{k\in\mathbb{Z}}|\theta_j-k|\geq 1/N$ for all $1\leq j\leq n$, so
\[1=d_1\bigl(\Z_o^{m},{\boldsymbol\theta-\boldsymbol 1}\bigr)\geq\sum_{j=1}^n\operatorname{dist}(\theta_j,\mathbb{Z})\geq n/N,\]
i.e., $N\geq n$. Thus
\[\sum_{j=1}^n b_j+\bigl(\eta^{-1}-2\bigr)\bigl(k'+k''\bigr)=\sum_{j=1}^n b_j+(N-2)\bigl(k'+k''\bigr)\geq n+n-2.\]
This proves (\ref{2200k3}).

{\bf Case 2.} $k'+k''=0$.
This case is not trivial.
Note that~\eqref{2200k3} is equivalent to
\begin{equation}\label{22200k3}2n-2\leq \sum_{j=1}^n b_j=\eta^{-1}\sum_{j=1}^n\theta_j.\end{equation}
Note that $1\leq b_1\leq b_2\leq\cdots\leq b_n$. If $b_3\geq 2$ then we have
\[\sum_{j=1}^n b_j\geq 1+1+2(n-2)=2n-2,\]
i.e., (\ref{22200k3}) holds. So we only need to consider the remaining case that $b_1=b_2=b_3=1$, i.e., $\eta=\theta_1=\theta_2=\theta_3$. Let $\ell\geq 3$ such that
\begin{equation}\label{bbbbb}b_1=\cdots=b_\ell=1,\qquad b_{\ell+1}\geq 2.\end{equation}
Now we claim that $b_n\geq \ell$. Once this claim is proved, then{\samepage
\[\sum_{j=1}^n b_j\geq \ell+2(n-\ell-1)+\ell=2n-2,\]
i.e., (\ref{22200k3}) holds and the proof is complete.}

Assume by contradiction that $b_n\leq \ell-1$.
Since $k'=k''=0$, we have
\[\sum_{j=n+1}^m \theta_{j}=m-2,\]
and there is $J_1\subset\{1,\dots,n\}$ such that $\sum_{j\in
 J_1}\theta_j=\sum_{j\in J_2}\theta_j$, where
$J_2=\{1,\dots,n\}\setminus J_1$. This implies $\sum_{j\in
 J_1}b_j=\sum_{j\in J_2}b_j$,
so
\begin{equation}\label{bbbeee}
\sum_{j=1}^n b_j \qquad \text{is even.}
\end{equation}

On the other hand, by~\eqref{200e0q-5} there is $(k_1,\dots,k_m)\in\mathbb{Z}^m$ such that $\sum_{j=1}^m k_j$ is odd and
\[\sum_{j=1}^m\bigl|\theta_j-1-k_j\bigr|=1.\]
Since $\theta_j\in\mathbb{Z}$ for $j\geq n+1$, we have $k_j=\theta_j-1$ for $j\geq n+1$ and so $\sum_{j=n+1}^m k_j=\sum_{j=n+1}^m (\theta_{j}-1)=m-2-(m-n)=n-2$. Denote $\tilde{k}_j=k_j+1$, then we obtain
\[\sum_{j=1}^n\big|\theta_j-\tilde{k}_j\big|=1\qquad \text{and}\qquad \sum_{j=1}^n \tilde{k}_j=\sum_{j=1}^m k_j+2\;\;\text{is odd}.\]
Note that $\theta_j=b_j\eta=b_j\theta_1$. Let $\theta_1=a+r$ with $a\in \mathbb{Z}$ and $r\in (0,1)$.

{\bf Case 2.1.} $0<r\leq 1/2$. Then
\[1=\sum_{j=1}^n\bigl|\theta_j-\tilde{k}_j\bigr|\geq \sum_{j=1}^\ell\bigl|\theta_1-\tilde{k}_j\bigr|\geq \ell r,\]
so $r\leq 1/\ell\leq 1/3$.
Consequently, for any $1\leq j\leq n$, $b_j\leq b_n\leq \ell-1$ implies
$\theta_j=b_j\theta_1=b_j a+ b_j r$ with $b_j r\in (0,1)$. From here and $1=\sum_{j=1}^n\bigl|\theta_j-\tilde{k}_j\bigr|$, we claim that
\begin{equation} \label{ccclll}
 \bigl|\theta_j-\tilde{k}_j\bigr|=b_jr\qquad\text{and so}\qquad \tilde{k}_j=b_j a,\qquad\forall j.
\end{equation}
Indeed,
we have $\bigl|\theta_j-\tilde{k}_j\bigr|\in \{b_jr, 1-b_jr\}$ for all $j$. Recalling~\eqref{bbbbb}, if there is $1\leq j_0\leq \ell$ such that \smash{$\bigl|\theta_{j_0}-\tilde{k}_{j_0}\bigr|=1-r$}, then
\[1\geq \sum_{j=1}^\ell\bigl|\theta_1-\tilde{k}_j\bigr|\geq (\ell-1)r+1-r>1,\]
a contradiction. Thus $\bigl|\theta_j-\tilde{k}_j\bigr|=b_jr =r$ for all $j\leq \ell$. Consequently, if $\bigl|\theta_{j_0}-\tilde{k}_{j_0}\bigr|=1-b_{j_0}r$ for some $\ell+1\leq j_0\leq n$, then it follows from $b_{j_0}\leq \ell-1$ that
\[1=\sum_{j=1}^n\bigl|\theta_j-\tilde{k}_j\bigr|\geq \ell r+1-b_{j_0}r>1,\]
a contradiction. This proves the claim (\ref{ccclll}) and so $a\sum_{j=1}^n b_j=\sum_{j=1}^n \tilde{k}_j$ is odd, a contradiction with (\ref{bbbeee}).

{\bf Case 2.2.} $1/2<r<1$. Then $\theta_1=a+1-(1-r)$
and
\[1=\sum_{j=1}^n\bigl|\theta_j-\tilde{k}_j\bigr|\geq \sum_{j=1}^\ell\bigl|\theta_1-\tilde{k}_j\bigr|\geq \ell (1-r),\]
so $1-r\leq 1/\ell\leq 1/3$. Consequently, for any $1\leq j\leq n$, $b_j\leq b_n\leq \ell-1$ implies
$\theta_j=b_j\theta_1=b_j (a+1)-b_j(1-r)$ with $b_j (1-r)\in (0,1)$. Then as in Case 2-1, we can prove that
\[
 \bigl|\theta_j-\tilde{k}_j\bigr|=b_j(1-r)\qquad\text{and so}\qquad \tilde{k}_j=b_j (a+1),\qquad\forall j,
\]
which implies so $(a+1)\sum_{j=1}^n b_j=\sum_{j=1}^n \tilde{k}_j$ is odd, again a contradiction with (\ref{bbbeee}).

The proof is complete.
\end{proof}

As mentioned in the introduction, the converse statement of Theorem~\ref{thm-MP-E} can not hold in general,
i.e.,~\eqref{00e0qq-111}--\eqref{k3} can not imply Mondello--Panov's
condition~\eqref{00e0q-5}. For example, let
$\{\theta_0,\theta_1,\theta_\infty,\theta_{t_1},\allowbreak\theta_{t_2}\}=\bigl\{\frac12,\frac12,\frac12,\frac32,\theta_{t_2}\bigr\}$
with $\theta_{t_3}\in 2\mathbb{N}+1$. Then it does not satisfy~\eqref{00e0q-5} but satisfy~\eqref{00e0qq-111}--\eqref{k3}, so co-axial solutions exist for some singular set $\{t_1, t_2\}$.
But this is not the case for $n=0$ (i.e., exactly three non-integer angles).

\section{Proof of Theorem~\ref{theorem: sharp}}
\label{section: proof of Theorem 1.4}

The goal of this section is to prove Theorem~\ref{theorem: sharp}.
Write $t_1$ and $\theta_{t_1}$ simply by $t$ and $\theta$.
The function $Q(x)$ in the differential
equation~\eqref{00eq: ODE on C} associated to the curvature equation
\eqref{00eq: DE on C} in the case under consideration is
\begin{equation} \label{eq: Heun Q}
Q(x)=\frac{\beta_0}{x^2}+\frac{d_0}x
+\frac{\beta_1}{(x-1)^2}+\frac{d_1}{x-1}
+\frac{\beta_t}{(x-t)^2}+\frac{d_t}{x-t},
\end{equation}
where $\beta_p=\alpha_p(\alpha_p+2)/4=\bigl(\theta_p^2-1\bigr)/4$ and $d$'s
satisfy~\eqref{e0q-32}, i.e.,
\begin{equation} \label{eq: d beta}
d_0+d_1+d_t=0, \qquad
\beta_0+\beta_1+\beta_t+d_1+td_t=\beta_\infty.
\end{equation}
Before we proceed further, we note that the set
$\{\alpha_p/2+1,-\alpha_p/2\}$ is invariant under the substitution
$\alpha_p\mapsto-2-\alpha_p=-\theta_p-1$. For simplicity of discussion later on,
we assume that~$\alpha_0$ and~$\alpha_1$ satisfy
\begin{equation} \label{eq: choice of alpha}
\alpha_0=\epsilon_0\theta_0-1, \qquad
\alpha_1=\epsilon_1\theta_1-1,
\end{equation}
where $\epsilon_0$, $\epsilon_1$ are given in the assumption of Theorem~\ref{theorem: sharp}. Thus, $\alpha_0+\alpha_1+\alpha_\infty=\epsilon_0\theta_0
+\epsilon_1\theta_1+\theta_\infty-3$ is an integer.

In view of~\eqref{eq: d beta}, we regard $d_0$ and $d_1$ as
functions of $d:=d_t$ and let $Q_{d,t}(x)$ denote the rational function
$Q(t)$ in~\eqref{eq: Heun Q}. By Theorem~\ref{prop: apparent}, there
exists a polynomial $\mathcal P(d,t)\in\Q_{\theta}[d,t]$ of degree
$\theta$ in $d$ such that the differential equation
\begin{equation} \label{eq: Heun equation}
 y''(x)=Q_{d,t}(x)y(x)
\end{equation}
is apparent at $t$ if and only if $\mathcal P(d,t)=0$, where
$\Q_{\theta}=\Q(\theta_0,\theta_1,\theta_\infty)$. The degree of
$\mathcal P(d,t)$ in~$d$ is~$\theta$, but its total degree is in
general strictly larger than $\theta$ (see~\eqref{e22q:PP}).
In our proof of Theorem~\ref{theorem: sharp}, we shall introduce
another pair $(\lambda,t)$ of unknowns in place of $(d,t)$.

Let
\begin{gather} \label{eq: lambda}
 \lambda:=td_0+t\frac{\alpha_0\alpha_1}2
 +\frac{\alpha_0\alpha_t}2
 =t(t-1)d_t+t(\beta_0+\beta_1+\beta_t-\beta_\infty)
 +t\frac{\alpha_0\alpha_1}2+\frac{\alpha_0\alpha_t}2,
\end{gather}
where the second equality follows from~\eqref{eq: d beta}. Indeed,
the parameter $\lambda$ is used as the accessary parameter of a
certain Heun equation considered in \cite{Quadrilateral}, where for
the simplicity of computation, the apparent singularity was put at
$x=0$ instead of $x=t$ in this paper, and the condition for
apparentness at $x=0$ was shown to be equivalent to the vanishing of
the characteristic polynomial of a finite Jacobi matrix (see~\cite[Proposition 2.4 and equation~(2.8)]{Quadrilateral}). For the
convenience of the reader, we recall the equivalence, stated in our
setting, as follows (see also~\mbox{\cite[Lemma B.8]{Chen-Lin-Yang}}):
\eqref{eq: Heun equation} is apparent at $t$ if and only if
 $\lambda$ is an eigenvalue of the $\theta\times\theta$ matrix
 \begin{equation} \label{eq: M}
 M=M(t):=\begin{pmatrix}
 B_1 & A_1 & & & & \\
 D_2 & B_2 & A_2 & & & \\
 & D_3 & B_3 & A_3 & & \\
 & & \ddots & \ddots & \ddots & \\
 & & & D_{\theta-1} & B_{\theta-1} & A_{\theta-1} \\
 & & & & D_{\theta} & B_{\theta} \end{pmatrix}
 \end{equation}
 (undisplayed entries are all $0$), where
 \begin{gather*}
 A_j =A_j(t):=t(t-1)j(j-\theta), \\
 B_j =B_j(t):=(2t-1)(j-1)(j-2) -(j-1)[(t-1)\alpha_0+t\alpha_1
 +(2t-1)\alpha_t]
 +t\alpha'\alpha'', \\
 D_j =D_j(t):=(j-2)(j-3)-(j-2)(\alpha_0+\alpha_1+\alpha_t)
 +\alpha'\alpha'',
 \end{gather*}
 with
\[
 \alpha':=-\frac{\alpha_0+\alpha_1+\alpha_t+\alpha_\infty}2-1,
 \qquad \alpha'':=\frac{\alpha_\infty-\alpha_0-\alpha_1-\alpha_t}2.
\]
 That is,~\eqref{eq: Heun equation} is apparent at $t$ if and only if
 $\lambda$ is an eigenvalue of $M(t)$, i.e., a zero of the polynomial
\[
 P(\lambda,t):=\det(\lambda I_\theta-M(t)).
\]
 The polynomial $P(\lambda,t)$ has the following properties.

\begin{Lemma} \label{lemma: P in lambda} \qquad
 \begin{enumerate}\itemsep=0pt
 \item[$({\rm i})$] The total degree and the degree in $\lambda$ of
 $P(\lambda,t)$ are both $\theta$.
 \item[$({\rm ii})$] The roots of $P(\lambda,0)$ are
 $(j-1)(\alpha_0+\alpha_t-j+2)$, $j=1,\dots,\theta$.
 \item[$({\rm iii})$] Let $P_\infty(\lambda):=\lim_{t\to\infty}P(\lambda
 t,t)/t^\theta$. Then the roots of $P_\infty(\lambda)$ are
 \begin{gather}
 \hat\lambda_j:=(j-1)(\alpha_\infty+\alpha_t-j+2)
 -\beta_\infty-\beta_t+\beta_0+\beta_1\nonumber\\
 \hphantom{\hat\lambda_j:=}{} +\frac{\alpha_0\alpha_1}2
 -\frac{\alpha_t\alpha_\infty}2,\qquad j=1,\dots,\theta.\label{eq: roots infinity}
 \end{gather}
 \end{enumerate}
\end{Lemma}

\begin{proof}
Statement (i) follows easily from the observations that
 $\deg_t A_j(t)=2$, ${\deg_tB_j(t)=1}$, and $D_j(t)$ are
 constant polynomials.

 The second property follows immediately from the fact that
 $A_j(0)=0$ and hence the roots of~$P(\lambda,0)$ are simply
 $B_j(0)$. The third property is proved in \cite[Theorem~B.3]{Chen-Lin-Yang}. Since the conclusions are stated in different
 ways, here we provide a sketch of proof.

 Let $u=1/x$. We check directly that $y(x)$ is a solution of~\eqref{eq: Heun equation} if and only if $\wt y(u):=uy(1/u)$
 satisfies
 \begin{equation} \label{eq: wt Heun}
 \frac{{\rm d}^2}{{\rm d}u^2}\wt y(u)=\wt Q(u)\wt y(u),
 \end{equation}
 where $\wt Q(u)=u^{-4}Q(1/u)$. Clearly, this is a Fuchsian
 differential equation with Riemann scheme
\[
 \begin{pmatrix}
 0 & 1 & 1/t & \infty \\
 -\alpha_\infty/2 & -\alpha_1/2 & -\alpha_t/2 &
 -(\alpha_0/2+1) \\
 \alpha_\infty/2+1 & \alpha_1/2+1 & \alpha_t/2+1
 & \alpha_0/2 \end{pmatrix}.
\]
 Thus,
\[
 \wt Q(u)=\frac{\beta_\infty}{u^2}+\frac{\wt d_0}u
 +\frac{\beta_1}{(u-1)^2}+\frac{\wt d_1}{u-1}
 +\frac{\beta_t}{(u-1/t)^2}+\frac{\wt d_{1/t}}{u-1/t}
\]
 for some complex numbers $\wt d_0$, $\wt d_1$, and $\wt d_{1/t}$
 satisfying
\[
 \wt d_0+\wt d_1+\wt d_{1/t}=0, \qquad
 \wt d_1+\frac{\wt d_{1/t}}t+\beta_\infty+\beta_1+\beta_t=\beta_0.
\]
In fact, computing the partial fraction decomposition of $\wt Q(u)$,
 we find that
\[
 \wt d_0=2\beta_1+2t\beta_t+d_1+t^2d_t, \qquad
 \wt d_1=-2\beta_1-d_1, \qquad
 \wt d_{1/t}=-2t\beta_t-t^2d_t.
\]
 We now apply parts (i) and (ii) to~\eqref{eq: wt Heun}. Let
\[
 \wt\lambda=\frac{\wt d_0}t+\frac{\alpha_\infty\alpha_1}{2t}
 +\frac{\alpha_\infty\alpha_t}2.
\]
 By parts (i) and (ii), there is a polynomial $\wt P\bigl(\wt\lambda,1/t\bigr)$
 such that~\eqref{eq: wt Heun} is apparent at $1/t$ if and only if
 $\wt P\bigl(\wt\lambda,1/t\bigr)=0$ and the roots of $\wt P\bigl(\wt\lambda,0\bigr)$
 are
\[
 (j-1)(\alpha_\infty+\alpha_t-j+2), \qquad
 j=1,\dots,\theta.
\]
 On the other hand, we see from~\eqref{eq: d beta} and~\eqref{eq: lambda} that
 \begin{align*}\wt d_0
 &{}=t(t-1)d_t+2t\beta_t+\beta_\infty+\beta_1-\beta_0-\beta_t\\
 &{}=\lambda+t(\beta_\infty+\beta_t-\beta_0-\beta_1)-t\frac{\alpha_0\alpha_1}{2}
 -\frac{\alpha_0\alpha_t}{2}+\beta_\infty+\beta_1-\beta_0-\beta_t, \end{align*}
 so $\wt\lambda$ and $\lambda$ are related by
 {\allowdisplaybreaks
 \begin{align*}
 \wt\lambda={}&\frac\lambda t+\beta_\infty+\beta_t-\beta_0-\beta_1
 -\frac{\alpha_0\alpha_1}2+\frac{\alpha_\infty\alpha_t}2 \\
 &{}+\frac1t(\beta_\infty+\beta_1-\beta_0-\beta_t)
 +\frac1{2t}(\alpha_1\alpha_\infty-\alpha_0\alpha_t).
 \end{align*}
 }%
 Letting $t\to\infty$, we conclude that the roots of
 $P_\infty(\lambda)$ are given by~\eqref{eq: roots infinity}.
\end{proof}

\begin{Corollary}
 \label{corollary: limiting equations}
 Let $\lambda$ and $\hat\lambda_j$ be given by~\eqref{eq: lambda} and~\eqref{eq: roots infinity}, respectively.
 Let $y_+(x;\lambda,t)$ and $y_-(x;\lambda,t)$ be solutions of~\eqref{eq: Heun equation} of the form
\begin{alignat*}{3}
 & y_+(x;\lambda,t)=x^{\alpha_0/2+1}\sum_{j=0}^\infty
 c_{+,j}(\lambda,t)x^j, \qquad && c_{+,0}(\lambda,t)=1,&
\\
 & y_-(x;\lambda,t)=x^{-\alpha_0/2}\sum_{j=0}^\infty
 c_{-,j}(\lambda,t)x^j, \qquad && c_{-,0}(\lambda,t)=1.&
\end{alignat*}
 Then as $t\to\infty$ with $\lambda/t\to\hat\lambda_j$,
 $y_\pm(x;\lambda,t)$ converge to solutions of the Fuchsian
 differential equation with Riemann scheme
 \begin{equation} \label{eq: RS infinity}
 \begin{pmatrix}
 0 & 1 & \infty \\
 -\alpha_0/2 & -\alpha_1/2 & -(\hat\alpha_\infty/2+1) \\
 \alpha_0/2+1 & \alpha_1/2+1 & \hat\alpha_\infty/2
 \end{pmatrix},
 \end{equation}
 where $\hat\alpha_\infty=\alpha_\infty+\alpha_t-2j+2$.
\end{Corollary}

\begin{proof} Let $\hat d_0$, $\hat d_1$, $\hat d_t$ be the limits of $d_0$,
 $d_1$, and $d_t$ as $t\to\infty$ and $\lambda/t\to\hat\lambda_j$.
 From~\eqref{eq: lambda}, we know that
\[
 \hat d_0=(j-1)(\alpha_\infty+\alpha_t-j+2)
 -\beta_\infty-\beta_t+\beta_0+\beta_1-\frac{\alpha_t\alpha_\infty}2
\]
 and $(t-1)d_t$ converge to finite numbers as $t\to\infty$ and
 $\lambda/t\to\hat\lambda_j$. The latter implies that
 $\hat d_t=0$. Thus, we have
\[
 \hat Q_j(x):=\lim_{t\to\infty,\,\lambda/t\to\hat\lambda_j}Q_{d,t}(x)
 =\frac{\beta_0}{x^2}+\frac{\hat d_0}x
 +\frac{\beta_1}{(x-1)^2}+\frac{\hat d_1}{x-1}.
\]
 From the relation $d_0+d_1+d_t=0$, we see that $\hat d_1=-\hat d_0$
 and
 \begin{equation*}
 \begin{split}
 \hat d_1+\beta_0+\beta_1
 &{}=-(j-1)(\alpha_\infty+\alpha_t-j+2)+\beta_\infty+\beta_t
 +\frac{\alpha_t\alpha_\infty}2 \\
 &{}=\frac14(\alpha_\infty+\alpha_t-2j+2)
 (\alpha_\infty+\alpha_t-2j+4).
 \end{split}
 \end{equation*}
 It follows that
\[
 y''(x)=\hat Q_j(x)y(x)
\]
 is a Fuchsian differential equation with Riemann scheme
 given by~\eqref{eq: RS infinity}. The convergence of~$Q_{d,t}(x)$ to~$\hat Q_j(x)$ is clearly uniform on any compact subset of ${\C\setminus\{0,1\}}$.
 Therefore, $y_\pm(x;\lambda,t)$ converge to solutions of
 $y''(x)=\hat Q_j(x)y(x)$. This completes the proof.
\end{proof}

Now for $p\in\{0,1,\infty,t\}$, let $M_p(\lambda,t)$ be the monodromy
matrices of~\eqref{eq: Heun equation} around $x=p$ with respect to the
basis $(y_+(x;\lambda,t),y_-(x;\lambda,t))^t$.
Recalling that the Riemann scheme of~\eqref{eq: Heun equation} is
\[
\begin{pmatrix}
 0 & 1 & t & \infty \\
 -\alpha_0/2 & -\alpha_1/2 & -\alpha_t/2 & -(\alpha_\infty/2+1) \\
 \alpha_0/2+1 & \alpha_1/2+1 & \alpha_t/2+1 & \alpha_\infty/2
\end{pmatrix},
\]
we know that
\[
M_0(\lambda,t)=\M{{\rm e}^{\pi {\rm i}\alpha_0}}00{{\rm e}^{-\pi {\rm i}\alpha_0}}, \qquad
M_t(\lambda,t)=(-1)^{\theta-1}I_2.
\]
Also, the monodromy matrices satisfy $M_0M_1M_tM_\infty=I_2$.

\begin{Lemma} We have
\[
M_1(\lambda,t)=\M{{\rm e}^{\pi {\rm i}\alpha_1}}{b(\lambda,t)}{c(\lambda,t)}
{{\rm e}^{-\pi {\rm i}\alpha_1}}
\]
for some complex numbers $b(\lambda,t)$ and $c(\lambda,t)$ satisfying
$b(\lambda,t)c(\lambda,t)=0$.
\end{Lemma}

\begin{proof} Write $M_1(\lambda,t)=\SM abcd$. Note that the
 eigenvalues of $M_1$ are ${\rm e}^{\pi {\rm i}\alpha_1}$ and ${\rm e}^{-\pi
 {\rm i}\alpha_1}$ and those of $M_\infty$ are ${\rm e}^{\pi {\rm i}\alpha_\infty}$
 and ${\rm e}^{-\pi {\rm i}\alpha_\infty}$. Together with the relation
 $M_0M_1M_tM_\infty=I_2$, these informations yield
 \begin{gather*}
 a+d={\rm e}^{\pi {\rm i}\alpha_1}+{\rm e}^{-\pi {\rm i}\alpha_1}, \\
 a{\rm e}^{\pi {\rm i}\alpha_0}+d{\rm e}^{-\pi {\rm i}\alpha_0}
 =(-1)^{\theta-1}\bigl({\rm e}^{\pi {\rm i}\alpha_\infty}+{\rm e}^{-\pi {\rm i}\alpha_\infty}\bigr).
 \end{gather*}
 Eliminating $d$, we get
\[
 \bigl({\rm e}^{\pi {\rm i}\alpha_0}-{\rm e}^{-\pi {\rm i}\alpha_0}\bigr)a
 =(-1)^{\theta-1}\bigl({\rm e}^{\pi {\rm i}\alpha_\infty}+{\rm e}^{-\pi {\rm i}\alpha_\infty}\bigr)
 -\bigl({\rm e}^{\pi {\rm i}(\alpha_1-\alpha_0)}{+{\rm e}^{\pi {\rm i}(-\alpha_1-\alpha_0)}}\bigr).
\]
 Now we have $\alpha_0+\alpha_1+\alpha_\infty=\epsilon_0\theta_0
 +\epsilon_1\theta_1+\theta_\infty-3=k-3$. Thus, the condition
 $\theta\equiv k\mod 2$ implies that
\[
 {\rm e}^{\pi {\rm i}(-\alpha_1-\alpha_0)}={\rm e}^{\pi {\rm i}(3-k+\alpha_\infty)}
 =(-1)^{\theta-1}{\rm e}^{\pi {\rm i}\alpha_\infty}
\]
 and hence $(-1)^{\theta-1}{\rm e}^{-\pi {\rm i}\alpha_\infty}
 ={\rm e}^{\pi {\rm i}(\alpha_0+\alpha_1)}$.
 From these, we see that $a={\rm e}^{\pi {\rm i}\alpha_1}$. It follows that
 $d={\rm e}^{-\pi {\rm i}\alpha_1}$ and $bc=0$. This proves the lemma.
\end{proof}

Let $P(\lambda,t)=P_1(\lambda,t)\cdots P_m(\lambda,t)$ be
the factorization of $P(\lambda,t)$ into a product of irreducible
polynomials over $\C$. It is easy to see that for each $j$, we have
either $b(\lambda,t)\equiv 0$ or $c(\lambda,t)\equiv 0$ on $\mathcal
A_j:=\bigl\{(\lambda,t)\in\C^2\colon P_j(\lambda,t)=0\bigr\}$. Thus, if we set
\[
 P_b(\lambda,t):=\prod_{b(\lambda,t)\equiv 0\text{ on }\mathcal A_j}
 P_j(\lambda,t), \qquad
 P_c(\lambda,t):=\prod_{c(\lambda,t)\equiv 0\text{ on }\mathcal A_j}
 P_j(\lambda,t),
\]
then the cardinality of the set $A$ is equal to the number of
intersections of the two curves $P_b(\lambda,t)=0$ and
$P_c(\lambda,t)=0$ in the affine plane $\C^2$, with the multiplicity
of a point in $A$ defined to be that of the corresponding intersection
point of $P_b$ and $P_c$. Furthermore, by part~(iii) of Lemma~\ref{lemma: P in lambda}, the two curves $P_b(\lambda,t)=0$ and
$P_c(\lambda,t)=0$ do not intersect at infinity. Therefore, by
Bezout's theorem, the number of intersections of the two curves in the
affine plane is $(\deg P_b)(\deg P_c)$. Hence, to prove Theorem~\ref{theorem: sharp}, we only need to determine the degrees of~$P_b$
and~$P_c$.

\begin{Proposition} \label{prop: degree of Pb}
 Let $k:=\theta_\infty+\epsilon_0\theta_0+\epsilon_1\theta_1$ be
 given as in the statement of Theorem~$\ref{theorem: sharp}$.
 If~$\theta\le|k|$, then
\[
 \begin{cases}
 \deg P_b=0\text{ and }\deg P_c=\theta, &\text{if }k>0, \\
 \deg P_b=\theta\text{ and }\deg P_c=0, &\text{if }k<0.
 \end{cases}
\]
 If $\theta>|k|$, then $\deg P_b=(\theta-k)/2$ and $\deg
 P_c=(\theta+k)/2$.
\end{Proposition}

To prove Proposition~\ref{prop: degree of Pb}, we recall that, by
Lemma~\ref{lemma: P in lambda}\,(i), the total degree and the degree in
$\lambda$ of $P(\lambda,t)$ are equal. It is easy to see that every
factor of $P(\lambda,t)$ has the same property. It follows that the
degree of $P_b(\lambda,t)$ is equal to the degree of the polynomial
$P_{b,\infty}(\lambda):=\lim_{t\to\infty}P_b(\lambda t,t)/t^{\deg
 P_b}$, which in turn is equal to the number of $\hat\lambda_j$ that
are roots of $P_{b,\infty}(\lambda)$, where $\hat\lambda_j$ are given
by~\eqref{eq: roots infinity}. In other words, to determine the degree
of $P_b(\lambda,t)$, we shall count how many $\hat\lambda_j$ such that
$\lim_{t\to\infty,\lambda/t\to\hat\lambda_j}c(\lambda,t)\neq 0$ there
are.

\begin{Lemma}\label{lem-111} Let $\hat\lambda_j$ be given by~\eqref{eq: roots
 infinity} and
\[
 \hat b=\lim_{t\to\infty,\lambda/t\to\hat\lambda_j}b(r,t), \qquad
 \hat c=\lim_{t\to\infty,\lambda/t\to\hat\lambda_j}c(r,t).
\]
 Then
\[
 \begin{cases}
 \hat b=0,~\hat c\neq 0, & \text{if }j>(k+\theta)/2, \\
 \hat b\neq0,~\hat c=0, &\text{if }j\le(k+\theta)/2.
 \end{cases}
\]
\end{Lemma}

\begin{proof} By Corollary~\ref{corollary: limiting equations},
 $y_\pm(x;\lambda,t)$ converge to solutions $\hat y_\pm(x)$ of the
 Fuchsian differential equation with Riemann scheme~\eqref{eq: RS
 infinity}, where $\hat y_\pm(x)$ are local solutions near $0$ of
 the form $\hat y_+(x)=x^{\alpha_0/2+1}(1+O(x))$ and $\hat
 y_-(x)=x^{-\alpha_0/2}(1+O(x))$. Then
 $w_{0,\pm}(x):=x^{\alpha_0/2}(x-1)^{\alpha_1/2}\hat
 y_\pm(x)$ are solutions of the hypergeometric differential equation
 with Riemann scheme
\[
 \begin{pmatrix} 0 & 1 & \infty \\
 0 & 0 & \alpha \\
 1-\gamma & \gamma-\alpha-\beta & \beta\end{pmatrix},
\]
 where
\[
 \gamma=-\alpha_0, \qquad
 \alpha=-\left(\frac{\alpha_0+\alpha_1+\hat\alpha_\infty}2+1\right),
 \qquad
 \beta=\frac{\hat\alpha_\infty-\alpha_0-\alpha_1}2.
\]
 Let $w_{1,\pm}(x)$ be the local solutions near $x=1$ of the same
 hypergeometric differential equation of the form
 $w_{1,+}(x)=(x-1)^{\gamma-\alpha-\beta}(1+O(x-1))$ and
 $w_{1,-}(x)=1+O(x-1)$. By~\cite[Chapter~2, Theorem~4.7.1]{Yoshida}, the connection matrix $P$ such that
 $(w_{0,+}(x),w_{0,-}(x))^t=P{(w_{1,+}(x),w_{1,-}(x))^t}$ is
\[
 P=\begin{pmatrix}
 \displaystyle
 \frac{\Gamma(2-\gamma)\Gamma(\gamma-\alpha-\beta)}
 {\Gamma(1+\alpha-\gamma)\Gamma(1+\beta-\gamma)} &
 \displaystyle
 \frac{\Gamma(2-\gamma)\Gamma(\gamma-\alpha-\beta)}
 {\Gamma(1-\alpha)\Gamma(1-\beta)} \\[10pt]
 \displaystyle
 \frac{\Gamma(\gamma)\Gamma(\gamma-\alpha-\beta)}
 {\Gamma(\alpha)\Gamma(\beta)} &
 \displaystyle\frac{\Gamma(\gamma)\Gamma(\gamma-\alpha-\beta)}
 {\Gamma(\gamma-\alpha)\Gamma(\gamma-\beta)} \end{pmatrix}.
\]
 Then the monodromy matrix of the hypergeometric differential
 equation around $x=1$ with respect to the basis
 $(w_{0,+}(x),w_{0,-}(x))^t$ is
\[
 P\M{{\rm e}^{2\pi {\rm i}(\gamma-\alpha-\beta)}}001P^{-1}.
\]
 Now according to our choice of $\alpha_0$ and $\alpha_1$ in~\eqref{eq: choice of alpha}, we have
 \begin{align*}
 \alpha&{}=-\left(\frac{\alpha_0+\alpha_1+\hat\alpha_\infty}2
 +1\right)
=-\frac{(\epsilon_0\theta_0-1)+(\epsilon_1\theta_1-1)
 +(\theta_\infty-1+\theta-1-2j+2)}2-1 \\
 &{}=j-\frac{k+\theta}2,
 \end{align*}
 which is an integer by the assumption $\theta\equiv k\mod 2$.
 Thus, $P$ is upper-triangular when $j\le(k+\theta)/2$ and is
 lower-triangular when $j>(k+\theta)/2$. This implies that $\hat
 b\neq0$ and $\hat c=0$ if $j\le(k+\theta)/2$, and $\hat b=0$ and
 $\hat c\neq 0$ if $j>(k+\theta)/2$.
\end{proof}

\begin{proof}[Proof of Proposition~\ref{prop: degree of Pb} and Theorem~\ref{theorem: sharp}] Since $j=1,2,\dots, \theta$, Proposition~\ref{prop: degree of Pb} follows from Lemma~\ref{lem-111}. Consequently, Theorem~\ref{theorem: sharp} holds.
\end{proof}

\begin{Remark} The entry $D_j$ in~\eqref{eq: M} can be expressed as
\[
 D_j=\frac14(2j-k-\theta-2)(2j-k-\theta+2\theta_\infty-2).
\]
 Thus, one has $D_j=0$ when $j=(k+\theta)/2+1$. It follows that when
 $|k|<\theta$, the matrix $M$ in~\eqref{eq: M} is of the form
\[
M=\M{M_1}*0{M_2},
\]
 where $M_1=M_1(t)$ and $M_2=M_2(t)$ are square matrices of size
 $(\theta+k)/2$ and $(\theta-k)/2$, respectively.
 Hence $P(\lambda,t):=\det(\lambda I-M)=\det(\lambda I-M_1)
 \det(\lambda I-M_2)$. One would expect that
 $P_c(\lambda,t)=\det(\lambda I-M_1)$ and
 $P_b(\lambda,t)=\det(\lambda I-M_2)$. Indeed, one can prove that
 this is the case by considering the roots of $P_b(\lambda,0)$
 and $P_c(\lambda,0)$. We will not give the details here.
\end{Remark}

For small $\theta$, it is possible to completely write down the set
$A$. We give some examples below.

\begin{Example}\label{example1} Assume that $\theta=2$. Then Theorem~\ref{theorem: sharp} implies that the set $A$ is non-empty
 only when $k=0$. In this case, we compute that
\[
 P(\lambda,t)=\lambda(\lambda-\theta_\infty t-\epsilon_0\theta_0).
\]
 The intersection of the two lines $\lambda=0$ and
 $\lambda-\theta_\infty t-\epsilon_0\theta_0=0$ is $(\lambda,t)
 =(0,-\epsilon_0\theta_0/\theta_\infty)$. Therefore, the set $A$
 consists of a single point $-\epsilon_0\theta_0/\theta_\infty$.

 Note that in \cite[Example 5.5]{Chen-Lin-Yang}, we have shown that
 for the case $\theta_0=1/3$, $\theta_1=1/6$, and
 $\theta_\infty=1/2$ with $\epsilon_0=\epsilon_1=-1$, the
 set $A$ is $\{2/3\}$. This agrees with the general result above.
\end{Example}

\begin{Example}\label{example2} Assume that $\theta=3$. Then Theorem~\ref{theorem: sharp} implies that the set $A$ is non-empty
 only when $|k|=1$. In this case,
 to simplify notations, we will write $\epsilon_0\theta_0$ and
 $\epsilon_1\theta_1$ simply by~$\theta_0$ and~$\theta_1$. We compute
 that when $k=1$,
 \begin{align*}
 P(\lambda,t)={}&(\lambda+(\theta_0+\theta_1)t-2\theta_0) \\
 &{}\times\bigl(\lambda^2-((\theta_0+\theta_1-2)t+\theta_0+1)\lambda
 +(\theta_0-1)(\theta_0+\theta_1)t\bigr),
 \end{align*}
 and when $k=-1$,
 \begin{align*}
 P(\lambda,t)={}&\lambda\bigl(\lambda^2+((3\theta_0+3\theta_1+2)t
 -3\theta_0-1)\lambda \\
 &{}+2(\theta_0+\theta_1)(\theta_0+\theta_1+1)t^2
 -4\theta_0(\theta_0+\theta_1+1)t+\theta_0(\theta_0+1)\bigr).
 \end{align*}
 We find that the set $A$ consists of
\[
 \frac{\theta_0\theta_\infty
 \pm\sqrt{-k\theta_0\theta_1\theta_\infty}}
 {(\theta_0+\theta_1)\theta_\infty}.
\]
 Note that the case $\theta_0=1/3$, $\theta_1=1/6$, and
 $\theta_\infty=1/2$ was considered in~\cite[Example~5.6]{Chen-Lin-Yang}, where we found that the set $A$ is $\{2(1\pm
{\rm i})/3\}$, as the result above says.
\end{Example}

\begin{Example} Let $\theta=4$. Again, we write $\epsilon_0\theta_0$
 and $\epsilon_1\theta_1$ simply by $\theta_0$ and $\theta_1$,
 respectively. When $k=2$, we find that the set
 $A$ consists of the three roots of
 \begin{equation*}
 \begin{split}
 \theta_\infty(\theta_\infty-1)(\theta_\infty-2)t^3
 +3\theta_0\theta_\infty(\theta_\infty-1)t^2
 +3\theta_\infty\theta_0(\theta_0-1)t
 +\theta_0(\theta_0-1)(\theta_0-2).
 \end{split}
 \end{equation*}
 When $k=-2$, they are the roots of
 \begin{equation*}
 \begin{split}
 \theta_\infty(\theta_\infty+1)(\theta_\infty+2)t^3
 +3\theta_0\theta_\infty(\theta_\infty+1)t^2
 +3\theta_\infty\theta_0(\theta_0+1)t
 +\theta_0(\theta_0+1)(\theta_0+2).
 \end{split}
 \end{equation*}
 When $k=0$, they are the roots of
 $$
 \theta_\infty^2\bigl(\theta_\infty^2-1\bigr)t^4
 +4\theta_0\theta_\infty\bigl(\theta_\infty^2-1\bigr)t^3
 +6\theta_0\theta_\infty(\theta_0\theta_\infty+1)t^2
 +4\theta_\infty\theta_0\bigl(\theta_0^2-1\bigr)t
 +\theta_0^2\bigl(\theta_0^2-1\bigr).
 $$
 (The polynomials $P(\lambda,t)$ themselves are too complicated to be
 displayed here.)
\end{Example}

\subsection*{Acknowledgements}

The authors thank the referees very much for careful reading and valuable comments. The research of the first author was supported by National Key R\&D Program of China (Grant 2022ZD0117000) and NSFC (No.~12222109, 12071240).

\pdfbookmark[1]{References}{ref}
\LastPageEnding

\end{document}